\documentclass[12pt,a4paper]{amsart}
\usepackage{amsmath}
\usepackage[all]{xy}
\usepackage{amssymb}
\usepackage[latin1]{inputenc}
\usepackage{amsart-color}
\usepackage{color}

\newtheorem{definition}{Definition}[section]
\newtheorem{lemma}[definition]{Lemma}
\newtheorem{theorem}[definition]{Theorem}
\newtheorem{proposition}[definition]{Proposition}

\newtheorem{remarkth}[definition]{Remark}
\newtheorem{example}[definition]{Example}
\newenvironment{remark}{\begin{remarkth}\upshape}{\hfill$\diamond$\end{remarkth}}
\renewenvironment{proof}[1][\proofname]%
  {\par\normalfont\topsep6pt plus 6pt\noindent{\scshape#1. \ignorespaces}}%
  {\qed\endtrivlist}

\renewcommand{\emph}[1]{{\bfseries\itshape{#1}}}

\newcommand{\map}[3]{#1\colon#2\rightarrow#3}
\newcommand{\set}[2]{\{#1\,|\,#2\}}
\newcommand{\Real}{\mathbb{R}}

\newcommand{\pd}[2]{\frac{{\partial#1}}{{\partial#2}}}
\newcommand{\id}{\operatorname{id}}
\renewcommand{\Im}{\operatorname{Im}}
\newcommand{\calu}{\mathcal{U}}
\newcommand{\calv}{\mathcal{V}}
\newcommand{\calw}{\mathcal{W}}

\def\lbeta{\mathsf{b}}
\def\lp{\mathsf{p}}

\let\epsilon\varepsilon

\let\lvec\overleftarrow
\let\rvec\overrightarrow

\newcommand{\lcf}{\lbrack\! \lbrack}
\newcommand{\rcf}{\rbrack\! \rbrack}
\newcommand{\F}{\mathbb{F}}
\newcommand{\R}{\mathbb{R}}
\begin{document}

\title[Local discrete Mechanics]{The local description of  discrete Mechanics}

\author[J.\ C.\ Marrero]{Juan C.\ Marrero}
\address{Juan C.\ Marrero:
Unidad Asociada ULL-CSIC Geometr{\'\i}a Diferencial y Mec\'anica
Geom\'etrica, Departamento de Matem\'atica Fundamental, Facultad
de Matem\'aticas, Universidad de la Laguna, La Laguna, Tenerife,
Canary Islands, Spain} \email{jcmarrer@ull.es}

\author[D.\ Mart\'\i n de Diego]{David Mart\'\i n de Diego}
\address{David Mart\'\i n de Diego:
Instituto de Ciencias Matem\'aticas, CSIC-UAM-UC3M-UCM,
Campus de Cantoblanco, UAM,
C/Nicol\'as Cabrera, 15
 28049 Madrid, Spain} 
\email{david.martin@icmat.es}

\author[E.\ Mart\'{\i}nez]{Eduardo Mart\'{\i}nez}
\address{Eduardo Mart\'{\i}nez:
Departamento de Matem\'atica Aplicada,
Facultad de Ciencias,
Universidad de Zaragoza,
50009 Zaragoza, Spain}
\email{emf@unizar.es}


\thanks{\noindent This work has been partially supported by MEC (Spain)
Grants   MTM2009-13383, MTM2010-21186-C02-01,  MTM2011-15725E, MTM2012-33575,  MTM2012-34478,  Aragon government project DGA-E24/1, project of the Canary Government Prod ID20100210,
 the ICMAT Severo Ochoa project SEV-2011-0087 and the European project IRSES-project ``Geomech-246981''.
JCM acknowledges the partial support from IUMA (University of
Zaragoza). JCM and DMdD have benefited from attending the conference {\sl Focus Program on Geometry, Mechanics and Dynamics, the legacy of Jerry Marsden}, held at the Fields Institute, so they thank the organizers. }


\begin{abstract}
In this paper, we introduce local expressions for discrete Mechanics. To apply our results  simultaneously to several interesting cases, we derive these local expressions in the framework of Lie groupoids, following the program proposed by  Alan Weinstein in \cite{We0}.
To do this,   we will need  some results on the geometry of Lie groupoids,
as, for instance,  the construction of symmetric neighborhoods or the existence
of local bisections. These local descriptions  will be particular useful
for the explicit construction of geometric integrators for mechanical systems (reduced or not),
in particular, discrete Euler-Lagrange equations, discrete Euler-Poincar\'e equations,
discrete Lagrange-Poincar\'e equations... The results contained in this paper can be
considered as a local version of the study that we have started in \cite{MaMaMa},
on the geometry of discrete Mechanics on Lie groupoids.
\end{abstract}

\maketitle

\begin{center}
{\sc Dedicated to the memory of J. E. Marsden}
\end{center}
\section{Introduction}

The use of geometrical methods in the study of dynamical systems (discrete or continuous) starts by searching for geometrical structures invariant with respect to the given dynamics. It turns out the various such structures emerge naturally for classical mechanical systems, as for instance,  symplectic or Poisson  structures, together with various bundle structures. Another geometrical feature that is common to all such systems is the presence of symmetries either because there is a redundant or extra information in the description of the system or because the system possesses an intrinsic invariance. 
To actually solve them, it is necessary in most occasions to use numerical methods. Recently, a new breed of ideas in numerical analysis have come that incorporates the geometry of the systems into the analysis and that allows  accurate and robust algorithms  with lower spurious effects than the traditional ones (see \cite{HaLuWa} and references therein).
 Our  approach employs the theory of discrete Mechanics and
variational integrators~\cite{mawest} to derive  an
integrator for the dynamics preserving some of the geometry of the original system.

The study of discrete Mechanics on Lie groupoids was proposed by A.~Weinstein in~\cite{We0}. This setting is general enough to include discrete counterparts of several types of fundamental equations in Mechanics as for instance, standard Euler-Lagrange equations for Lagrangians defined on tangent bundles, Euler-Poincar\'e equations for Lagrangians defined on Lie algebras, Lagrange-Poincar\'e equations for Lagrangians defined on Atiyah bundles, etc.
Such discrete counterpart is obtained by discretizing the continuous Lagrangian to the corresponding Lie groupoid and then applying a discrete variational derivation of the discrete equations of motion. As simple examples, for a given differentiable manifold  $Q$, the discrete version of the tangent bundle $TQ$ is the product manifold $Q\times Q$, equipped with the pair Lie groupoid structure; for a given Lie group $G$,  the discrete version of its Lie algebra ${\mathfrak g}$ is the Lie group $G$.

A  Lie groupoid $G$ is a natural generalization of the concept of a Lie group, where now not all  elements are composable. The product $g_1 g_2$ of two elements is only defined on the set of composable pairs $G_2=\set{(g, h)\in G\times G}{\beta(g)=\alpha(h)}$ where
$\alpha: G\longrightarrow M$ and $\beta: G\longrightarrow M$ are
the source and target maps over a base manifold $M$. Moreover, in a Lie groupoid we have a set of identities playing a similar role that the identity element in group theory.
 The infinitesimal counterpart of the notion of a Lie groupoid is the notion of a Lie algebroid $\tau_{AG}: AG\to M$,
in the same way as the infinitesimal counterpart of the notion of a Lie group is the notion of a Lie algebra, or in other words, the discrete version of a Lie algebroid is a Lie groupoid.

In \cite{MaMaMa} we have elucidated the geometrical framework for
discrete Mechanics on Lie groupoids. In that paper, we found intrinsic expressions
for the discrete Euler-Lagrange equations, and we have introduced the
Poincar\'e-Cartan sections,  the discrete Legendre transformations
and the discrete evolution operator in both the Lagrangian and the
Hamiltonian formalism. The notion of regularity has been
completely characterized and we have proven  the symplecticity of the
discrete evolution operators. The
applicability of these developments has been stated in several
interesting examples, in particular for the case of discrete
Lagrange-Poincar\'e equations. In fact, the general theory of
discrete symmetry reduction directly follows from our results.

The main objective of this paper is to obtain local expressions for the different objects appearing in discrete Mechanics on Lie groupoids. For this proposal, it is necessary to introduce symmetric neighborhoods of a Lie groupoid. A symmetric neighborhood is an open neighborhood of one point in the manifold of the identities which  is ``natural" with respect to the structure maps of the Lie groupoid, in the sense of Proposition \ref{sym}.
Using the coordinates associated to a symmetric neighborhood we may write the local expressions of left and right invariant vector fields, $\lvec{X}$ and $\rvec{X}$ associated to a section $X\in \Gamma (\tau_{AG})$ of the associated Lie algebroid. Now, as we have deduced in \cite{LeMaMa}, the discrete Euler-Lagrange equations for a discrete Lagrangian $L_d: G\to \R$  are
\[
\lvec{X}(g_k)(L_d)-\rvec{X}(g_{k+1})(L_d)=0\; ,\quad (g_k, g_{k+1})\in G_2\; .
\]
Therefore, from this expression, we easily obtain the local expression of the discrete Euler-Lagrange equations associated to a discrete Lagrangian $L_d: G\to \R$, the discrete Legendre transformations and we locally characterize the regularity of the discrete problem.

An interesting point is that when using symmetric neighborhoods we are implicitly assuming that the discrete flow is well defined on this neighborhood. However, this is not the more general situation since, in principle, the point $g_k$ and its image $g_{k+1}$ under the discrete flow may be far enough in such a way both are not included in the same symmetric neighborhood. In order to tackle this problem we will use bisections of the Lie groupoid which permits to translate neighborhoods of two composable elements $g_k$ and $g_{k+1}$ to a symmetric neighborhood at the identity point $\beta(g_k)=\alpha(g_{k+1})$.

The organization of the paper is as follows. In section~2 we
recall some constructions and results on discrete Mechanics on Lie
groupoids which will be used in the next sections. In section~3,
we will obtain a local expression of the discrete Euler-Lagrange
equation for a discrete Lagrangian function on a symmetric
neighborhood in the Lie groupoid where it is defined. In addition
we will discuss the existence of local discrete Euler-Lagrange
evolution operators in such a symmetric neighborhood.  Moreover, several interesting examples are considered. The
existence of general local discrete Euler-Lagrange evolution
operators is studied in section~4. For this purpose, we will use
bisections on the Lie groupoid. 
 The paper ends with our conclusions and a
description of future research directions.

\section{Groupoids and discrete Mechanics}

\subsection{Lie groupoids}\label{section2.2}

In this Section, we will recall the definition of a Lie groupoid
and some generalities about them are explained (for more details,
see \cite{CaWe,Mac}).

A \emph{groupoid} over a set $M$ is a set $G$ together with the
following structural maps:
\begin{itemize}
\item A pair of maps $\alpha: G \to M$, the \emph{source}, and
$\beta: G \to M$, the \emph{target}. Thus, an element $g \in G$ is
thought as an arrow from $x= \alpha(g)$ to $y = \beta(g)$ in $M$
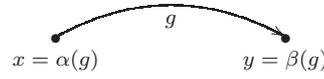
$$
\xymatrix{*=0{\stackrel{\bullet}{\mbox{\tiny
 $x=\alpha(g)$}}}{\ar@/^1pc/@<1ex>[rrr]_g}&&&*=0{\stackrel{\bullet}{\mbox{\tiny
$y=\beta(g)$}}}}
$$
The maps $\alpha$ and $\beta$ define the set of composable pairs
$$
G_{2}=\{(g,h) \in G \times G / \beta(g)=\alpha(h)\}.
$$
\item A \emph{multiplication} $m: G_{2} \to G$, to be denoted simply by $m(g,h)=gh$, such that
\begin{itemize}
\item $\alpha(gh)=\alpha(g)$ and $\beta(gh)=\beta(h)$.
\item $g(hk)=(gh)k$.
\end{itemize}
If $g$ is an arrow from $x = \alpha(g)$ to $y = \beta(g) =
\alpha(h)$ and $h$ is an arrow from $y$ to $z = \beta(h)$ then
$gh$ is the composite arrow from $x$ to $z$
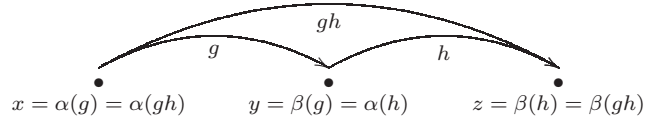
$$\xymatrix{*=0{\stackrel{\bullet}{\mbox{\tiny
 $x=\alpha(g)=\alpha(gh)$}}}{\ar@/^2pc/@<2ex>[rrrrrr]_{gh}}{\ar@/^1pc/@<2ex>[rrr]_g}&&&*=0{\stackrel{\bullet}{\mbox{\tiny
 $y=\beta(g)=\alpha(h)$}}}{\ar@/^1pc/@<2ex>[rrr]_h}&&&*=0{\stackrel{\bullet}{\mbox{\tiny
 $z=\beta(h)=\beta(gh)$}}}}$$
\item An \emph{identity map} $\epsilon: M \to G$, a section of $\alpha$ and $\beta$, such that
\begin{itemize}
\item $\epsilon(\alpha(g))g=g$ and $g\epsilon(\beta(g))=g$.
\end{itemize}
\item An \emph{inversion map} $i: G \to G$, to be denoted simply by $i(g)=g^{-1}$, such that
\begin{itemize}
\item $g^{-1}g=\epsilon(\beta(g))$ and $gg^{-1}=\epsilon(\alpha(g))$.
\end{itemize}
$$\xymatrix{*=0{\stackrel{\bullet}{\mbox{\tiny
 $x=\alpha(g)=\beta(g^{-1})$}}}{\ar@/^1pc/@<2ex>[rrr]_g}&&&*=0{\stackrel{\bullet}{\mbox{\tiny
 $y=\beta(g)=\alpha(g^{-1})$}}}{\ar@/^1pc/@<2ex>[lll]_{g^{-1}}}}$$

\end{itemize}

A groupoid $G$ over a set $M$ will be denoted simply by the symbol
$G \rightrightarrows M$.

The groupoid $G \rightrightarrows M$ is said to be a \emph{Lie
groupoid} if $G$ and $M$ are manifolds and all the structural maps
are differentiable with $\alpha$ and $\beta$ differentiable
submersions. If $G \rightrightarrows M$ is a Lie groupoid then $m$
is a submersion, $\epsilon$ is an immersion and $i$ is a
diffeomorphism. Moreover, if $x \in M$, $\alpha^{-1}(x)$ (resp.,
$\beta^{-1}(x)$) will be said the \emph{$\alpha$-fiber} (resp.,
the \emph{$\beta$-fiber}) of $x$.

On the other hand, if $g \in G$ then the \emph{left-translation by
$g \in G$} and the \emph{right-translation by $g$} are the
diffeomorphisms
$$
\begin{array}{lll}
l_{g}: \alpha^{-1}(\beta(g)) \longrightarrow
\alpha^{-1}(\alpha(g))&; \; \;& h \longrightarrow
l_{g}(h) = gh, \\
r_{g}: \beta^{-1}(\alpha(g)) \longrightarrow
\beta^{-1}(\beta(g))&; \; \;& h \longrightarrow r_{g}(h) = hg.
\end{array}
$$
Note that $l_{g}^{-1} = l_{g^{-1}}$ and $r_{g}^{-1} = r_{g^{-1}}$.

A vector field $\tilde{X}$ on $G$ is said to be
\emph{left-invariant} (resp., \emph{right-invariant}) if it is
tangent to the fibers of $\alpha$ (resp., $\beta$) and
$\tilde{X}(gh) = (T_{h}l_{g})(\tilde{X}_{h})$ (resp.,
$\tilde{X}(gh) = (T_{g}r_{h})(\tilde{X}(g)))$, for $(g,h) \in
G_{2}$.

Now, we will recall the definition of the \emph{Lie algebroid
associated with $G$}.

We consider the vector bundle $\tau: AG \to M$, whose fiber at a
point $x \in M$ is $A_{x}G = V_{\epsilon(x)}\alpha = Ker
(T_{\epsilon(x)}\alpha)$. It is easy to prove that there exists a
bijection between the space $\Gamma (\tau)$ and the set of
left-invariant (resp., right-invariant) vector fields on $G$. If
$X$ is a section of $\tau: AG \to M$, the corresponding
left-invariant (resp., right-invariant) vector field on $G$ will
be denoted $\lvec{X}$ (resp., $\rvec{X}$), where
\begin{equation}\label{linv}
\lvec{X}(g) = (T_{\epsilon(\beta(g))}l_{g})(X(\beta(g))),
\end{equation}
\begin{equation}\label{rinv}
\rvec{X}(g) = -(T_{\epsilon(\alpha(g))}r_{g})((T_{\epsilon
(\alpha(g))}i)( X(\alpha(g)))),
\end{equation}
for $g \in G$. Using the above facts, we may introduce a Lie
algebroid structure $(\lcf\cdot , \cdot\rcf, \rho)$ on $AG$, which
is defined by
\begin{equation}\label{LA}
\lvec{\lcf X, Y\rcf} = [\lvec{X}, \lvec{Y}], \makebox[.3cm]{}
\rho(X)(x) = (T_{\epsilon(x)}\beta)(X(x)),
\end{equation}
for $X, Y \in \Gamma(\tau)$ and $x \in M$. Note that
\begin{equation}\label{RL}
\rvec{\lcf X, Y\rcf} = -[\rvec{X}, \rvec{Y}], \makebox[.3cm]{}
[\rvec{X}, \lvec{Y}] = 0,
\end{equation}
\begin{equation}\label{2.6'}
Ti\circ \rvec{X}=-\lvec{X}\circ i,\;\;\;\; Ti\circ
\lvec{X}=-\rvec{X}\circ i,
\end{equation}
(for more details, see \cite{CoDaWe,Mac}).

\subsection{Discrete Euler-Lagrange equations}
Let $G$ be  a Lie groupoid with structural maps
\[
\alpha, \beta: G \to M, \; \; \epsilon: M \to G, \; \; i: G \to G,
\; \; m: G_{2} \to G.
\]
Denote by $\tau:AG\to M$ the Lie algebroid of $G$.

 A \emph{discrete Lagrangian} is a function $\map{L_d}{G}{\Real}$. Fixed $g\in
G$, we define the set of admissible sequences with values in $G$:
\[
\begin{array}{rcl}
{\mathcal C}^N_{g}=\{(g_1, \ldots, g_N)\in G^N\; / \; (g_k,
g_{k+1})\in G_2 \hbox{ for } k=1,\ldots, N-1 \\ \hbox{ and } g_1
\ldots g_n=g  \}.
\end{array}
\]
An admissible sequence $(g_{1}, \dots , g_{N}) \in {\mathcal
C}^N_g$  is a solution of \emph{the discrete Euler-Lagrange
equations} if
\[
0=\sum_{k=1}^{N-1}\left[\lvec{X}_k\big({g_k})(L_d)-\rvec{X}_k\big({g_{k+1}})(L_d)
\right], \; \; \mbox{ for } X_{1}, \dots , X_{N-1} \in
\Gamma(\tau).
\]
For $N=2$ we obtain that $(g, h)\in G_2$ is a solution if
\[
\lvec{X}({g})(L_d)-\rvec{X}({h})(L_d)=0
\]
for every section $X$ of $AG$.

\subsection{Discrete Poincar\'e-Cartan sections}

Given a Lagrangian function $\map{L_d}{G}{\Real}$, we will study
the geometrical properties of the discrete Euler-Lagrange
equations.

Consider the vector bundle
\[
\pi^{\tau}: P^{\tau}G = V\beta \oplus V\alpha \to G
\]
where $V\alpha$ (respectively, $V\beta$) is the vertical bundle of
the source map $\alpha: G \to M$ (respectively, the target map
$\beta: G \to M$). Then, one may introduce a Lie algebroid
structure on $\pi^{\tau}: P^{\tau}G = V\beta \oplus V\alpha \to
G$. The anchor map $\rho^{P^{\tau}G}: P^{\tau}G = V\beta \oplus
V\alpha \to TG$ is given by
\[
\rho^{P^{\tau}G}(X_{g}, Y_{g}) = X_{g} + Y_{g}, \; \; \mbox{ for }
(X_{g}, Y_{g}) \in V_{g}\beta \oplus V_{g}\alpha
\]
and the Lie bracket $\lcf\cdot , \cdot\rcf^{{\mathcal P}^{\tau}G}$
on the space $\Gamma(\pi^\tau)$ is characterized by the following
relation
\begin{equation}\label{e1''}
\lcf (\rvec{X}, \lvec{Y}), (\rvec{X'}, \lvec{Y'}) \rcf ^{{\mathcal
P}^{\tau}G} = (-\overrightarrow{\lcf X, X' \rcf},
\overleftarrow{\lcf Y, Y' \rcf}),
\end{equation}
for $X, Y, X', Y' \in \Gamma(\tau)$ (see \cite{MaMaMa}).

Now, define the \emph{Poincar\'e-Cartan 1-sections }
$\Theta_{L_{d}}^-, \Theta_{L_{d}}^+\in \Gamma ((\pi^\tau)^*)$ as
follows
\begin{equation}\label{5.16'}
\Theta_{L_{d}}^-(g)(X_g, Y_g)= -X_g(L_{d}),\;\;\;\;\;
\Theta_{L_{d}}^+(g)(X_g, Y_g)= Y_g(L_{d}),
\end{equation}
 for each $g\in G$ and $(X_g, Y_g)\in V_g\beta\oplus
V_g\alpha$.

If $d$ is the differential of the Lie algebroid $\pi^{\tau}:
P^{\tau}G = V\beta \oplus V\alpha \to G$ we have that
$dL_d=\Theta_{L_{d}}^+ - \Theta_{L_{d}}^- $ and so, using $d^2=0,$
it follows  that $d\Theta_{L_{d}}^+=d\Theta_{L_{d}}^-$. This means
that there exists a unique 2-section
$\Omega_{L_{d}}=-d\Theta_{L_{d}}^+=-d\Theta_{L_{d}}^-$, that will
be called the \emph{Poincar\'e-Cartan} 2-section. This 2-section
will be important for studying  symplecticity of the discrete
Euler-Lagrange equations.

Let $X$ be a section of the Lie algebroid $\tau: AG \to M$. Then,
one may consider the sections $X^{(1,0)}$ and $X^{(0,1)}$ of the
vector bundle $\pi^{\tau}: P^{\tau}G = V\beta \oplus V\alpha \to
G$ given by
\[
X^{(1,0)}(g) = (\rvec{X}(g), 0_{g}), \; \; \; X^{(0,1)}(g) =
(0_{g}, \lvec{X}(g)), \; \; \mbox{ for } g \in G.
\]
Moreover, if $g\in G$, $\{X_{\gamma}\}$ (respectively,
$\{Y_{\mu}\}$) is a local basis of $\Gamma(\tau)$ in an open
subset $U$ (respectively, $V$) of $M$ such that $\alpha(g) \in U$
(respectively, $\beta(g) \in V$) then $\{X_{\gamma}^{(1,0)},
Y_{\mu}^{(0,1)}\}$ is a local basis of $\Gamma(\pi^{\tau})$ in
$\alpha^{-1}(U) \cap \beta^{-1}(V)$ and
\begin{equation}\label{Omega1}
\Omega_{L_{d}}(X_{\gamma}^{(1,0)}, Y_{\mu}^{(1,0)}) =
\Omega_{L_{d}}(X_{\gamma}^{(0,1)}, Y_{\mu}^{(0,1)}) = 0,
\end{equation}
and
\begin{equation}\label{Omega2}
\Omega_{L_{d}}(X_{\gamma}^{(1,0)}, Y_{\mu}^{(0,1)}) =
\lvec{Y_{\mu}}(\rvec{X_{\gamma}}(L_{d})) =
\rvec{X_{\gamma}}(\lvec{Y_{\mu}}(L_{d})).
\end{equation}
(for more details, see \cite{MaMaMa}).

\subsection{Discrete Lagrangian evolution operator}

We say that a differentiable mapping $\Psi: G\longrightarrow G$ is
a \emph{discrete flow} or a \emph{discrete Lagrangian evolution
operator for $L_{d}$} if it verifies the following properties:
\begin{enumerate}
\item[-] $\hbox{graph}(\Psi)\subseteq G_2$, that is, $(g, \Psi(g))\in G_2$, $\forall g\in
G$.
\item[-] $(g, \Psi(g))$ is a solution of the discrete Euler-Lagrange
equations, for all $g\in G$, that is,
\begin{equation}\label{5.22'}
\lvec{X}(g)(L_d)-\rvec{X}(\Psi(g))(L_d)=0
\end{equation}
for every section $X$ of $AG$ and every $g\in G.$
\end{enumerate}

\subsection{Discrete Legendre transformations}
\label{section5.6} Given a discrete La\-gran\-gian
$\map{L_{d}}{G}{\Real}$ we define two \emph{discrete Legendre
transformations} $\F^{-}L_{d}: G\longrightarrow A^*G$ and
$\F^{+}L_{d}: G\longrightarrow A^*G$
 as follows (see \cite{MaMaMa})
\begin{equation}\label{DLt-}
(\F^{-}L_{d})(h)(v_{\epsilon(\alpha(h))})=-v_{\epsilon(\alpha(h))}(L_d\circ
r_h\circ i), \mbox{ for } v_{\epsilon(\alpha(h))}\in
A_{\alpha(h)}G,
\end{equation}
\begin{equation}\label{DLt+}
(\F^{+}L_d)(g)(v_{\epsilon(\beta(g))})=
v_{\epsilon(\beta(g))}(L_d\circ l_g), \mbox{ for }
v_{\epsilon(\beta(g))}\in A_{\beta(g)}G.
\end{equation}
\begin{remark}\label{r4.4'}
Note that  $(\F^{+}L_d)(g)\in A^*_{\beta(g)}G$ and $(\F^{-}L_d)(h)\in
A^*_{\alpha(h)}G$. Furthermore, if $\{X_\gamma\}$ (respectively,
$\{Y_{\mu}\}$) is a local basis of $\Gamma(\tau)$ in an open subset $U$ such that
$\alpha(h) \in U$ (respectively, $\beta(g) \in V$) and
$\{X^\gamma\}$ (respectively, $\{Y^\mu\}$) is the dual basis of
$\Gamma(\tau^*),$ it follows that
\[
\F^-L_{d}(h)=\rvec{X}_\gamma(h)(L_d)X^\gamma(\alpha(h)),\;\;\;
\F^+L_{d}(g)=\lvec{Y}_\mu(g)(L_d)Y^\mu(\beta(g)).
\]
\end{remark}

\subsection{Discrete regular Lagrangians}
A Lagrangian $L_d: G\to \Real$ on a Lie groupoid $G$ is said to be
\emph{regular} if the Poincar{\'e}-Cartan $2$-section
$\Omega_{L_{d}}$ is symplectic on the Lie algebroid $\pi^{\tau}:
P^\tau G\equiv V\beta\oplus_GV\alpha\to G$, that is,
$\Omega_{L_{d}}$ is nondegenerate (see \cite{MaMaMa}).

Using (\ref{Omega2}), we deduce that the Lagrangian $L_{d}$ is
regular if and only if for every $g\in G$ and every local basis
$\{X_\gamma\}$ (respectively, $\{Y_\mu\}$) of $\Gamma(\tau)$ on an
open subset $U$ (respectively, $V$) of $M$ such that $\alpha(g)\in
U$ (respectively, $\beta(g)\in V$) we have that the matrix
$\rvec{X_\gamma}(\lvec{Y_\mu}(L_d))$ is regular on
$\alpha^{-1}(U)\cap \beta^{-1}(V)$.

In \cite{MaMaMa}, we have proved that the following conditions are
equivalent:
\begin{itemize}
\item $L_d: G \to \Real$ is a regular discrete Lagrangian function.

\item The Legendre transformation $\F^-L_d$ is a local
diffeomorphism.

\item The Legendre transformation $\F^+L_d$ is a local
diffeomorphism.
\end{itemize}

Moreover, if $L_d:G\to \Real$ is regular and $(g_0,h_0)\in G_{2}$
is a solution of the discrete Euler-Lagrange equations for $L_d$
then there exist two open subsets $U_0$ and $V_0$ of $G$, with
$g_0\in U_0$ and $h_0\in V_0,$ and there exists a (local) discrete
Lagrangian evolution operator $\Psi_{L_{d}}:U_0\to V_0$ such that:
\begin{itemize}
\item $\Psi_{L_{d}}(g_0)=h_0,$

\item $\Psi_{L_{d}}$ is a diffeomorphism and

\item $\Psi_{L_{d}}$ is unique, that is, if $U_0'$ is an open subset of
$G$, with $g_0\in U_0'$ and $\Psi_{L_{d}}':U'_0\to G$ is a (local)
discrete Lagrangian evolution operator then $
\Psi_{L_{d}}'{|U_0\cap U_0'}=\Psi_{L_{d}}{|U_0\cap U_0'}$.
\end{itemize}

\section{Discrete Euler-Lagrange equations: symmetric neighborhoods}\label{section:symmetric}

\subsection{Symmetric neighborhoods}

First, we prove the following result
\begin{proposition}\label{sym}
Let $\calu$ be an open subset of $G$ and $x_0\in M$ be a point
such that $\epsilon(x_0)\in\calu$. There exists an open subset
$\calw\subset\calu$ of $G$ with $\epsilon(x_0)\in\calw$ and such
that
\begin{enumerate}
\item $\epsilon(\alpha(\calw))\subset\calw$ and $\epsilon(\beta(\calw))\subset\calw$,
\item $i(\calw)=\calw$, and
\item $m((\calw\times\calw)\cap G_2)\subset\calu$.
\end{enumerate}
The open subset $\calw$ is said to be a \emph{symmetric
neighborhood} associated to $\calu$ and $x_0$.
\end{proposition}
\begin{proof}
The multiplication map $\map{m}{G_2}{G}$ is continuous, so that we
may choose an open subset $\calw_1$ of $G$ such that
$\epsilon(x_0)\in\calw_1$ and $m((\calw_1\times\calw_1)\cap
G_2)\subset \calu$. Since the identity map $\map{\epsilon}{M}{G}$
is also continuous, we deduce that there exists an open subset
$\calv'$ of $M$ such that $x_0\in\calv'$ and
$\epsilon(\calv')\subset\calw_1$. Thus, if we consider the open
$\calw_2$ of $G$ given by
$\calw_2=\calw_1\cap\alpha^{-1}(\calv')\cap\beta^{-1}(\calv')$,
then it is clear that $\epsilon(x_0)\in\calw_2$ and moreover it is
easy to prove that $\epsilon(\alpha(\calw_2))\subset\calw_2$ and
$\epsilon(\beta(\calw_2))\subset\calw_2$, and also
$m((\calw_2\times\calw_2)\cap G_2)\subset\calu$. Finally, if we
take $\calw=\calw_2\cap i(\calw_2)$ it follows that $\calw$
satisfies the three above mentioned conditions.
\end{proof}

\subsection{Local coordinate expressions of structural
maps}\label{local-expressions}
 On a symmetric neighborhood of a
point it is easy to get local coordinate expressions for the
structure maps of the Lie groupoid $G$. We consider a point
$x_0\in M$ and a local coordinate system $(x,u)$, defined in a
neighborhood $\calu\subset G$ of $\epsilon(x_0)$,  adapted to the
fibration $\map{\alpha}{G}{M}$, i.e. if the coordinates of
$g\in\calu$ are $(x^i,u^\gamma)$ then the coordinates of
$\alpha(g)\in M$ are $(x^i)$. We can moreover assume that the
identities correspond to elements with coordinates $(x,0)$. The
target map $\beta$ defines a local function $\lbeta$ as follows:
if the coordinates of $g$ are $(x,u)$, then  the coordinates of
$\beta(g)$ are $\lbeta(x,u)$. Note that $\lbeta(x,0) = x$. Two
elements $g$ and $h$ with coordinates $(x,u)$ and $(y,v)$ are
composable if and only if $y=\lbeta(x,u)$. Hence local coordinates
for $G_2$ are given by $(x,u,v)$.

To obtain a local description for the product, we consider a
symmetric neighborhood $\calw$ associated to $x_0$ and $\calu$. If
two elements $g, h\in\calw$ with coordinates $(x,u)$ and $(y,v)$
respectively, are composable then $y=\lbeta(x,u)$, and the product
$gh$ has coordinates $(x,\lp(x,u,v))$ for some smooth function
$\lp$. We will write
\begin{equation}
(x,u)\cdot(y,v)=(x,\lp(x,u,v)).
\end{equation}
The relation $\beta(gh)=\beta(h)$, for $(g,h)\in G_2$, imposes the restriction $\lbeta(y,v)=\lbeta(x,\lp(x,u,v))$, i.e.
\begin{equation}
\label{beta gh=beta h}
\lbeta(\lbeta(x,u),v)=\lbeta(x,\lp(x,u,v)).
\end{equation}
The property $g\epsilon(\beta(g))=g$, for $g\in\calw$, is locally equivalent to the equation $\lp(x,u,0)=u$; while the property  $\epsilon(\alpha(g))g=g$ is locally equivalent to the equation $\lp(x,0,v)=v$. Therefore
\begin{equation}
\label{p(x,u,0)=u}
\lp(x,u,0)=u,\qquad \lp(x,0,v)=v.
\end{equation}
Associativity $(gh)k=g(hk)$ imposes the further relation
\begin{equation}
\lp(x,\lp(x,u,v),w)=\lp(x,u,\lp(y,v,w))\qquad\text{with $y=\lbeta(x,u)$.}
\end{equation}

In what follows we will use the following functions defined in terms of $\lbeta(x,u)$ and $\lp(x,u,v)$,
\begin{equation}\label{rho-L-R}
\begin{aligned}
&\rho^i_\gamma(x)=\pd{\lbeta^i}{u^\gamma}(x,0)\\
&L^\gamma_\mu(x,u)=\pd{\lp^\gamma}{v^\mu}(x,u,0)\\
&R^\gamma_\mu(x,v)=\pd{\lp^\gamma}{u^\mu}(x,0,v).
\end{aligned}
\end{equation}
We will also take into account that
\begin{equation}
\label{pd and pd2 of p}
\begin{aligned}
&\pd{\lp^\gamma}{u^\mu}(x,u,0)=\delta^\gamma_\mu
\qquad\qquad&&\pd{^2\lp^\gamma}{u^\mu\partial u^\nu}(x,u,0)=0\\
&\pd{\lp^\gamma}{v^\mu}(x,0,v)=\delta^\gamma_\mu
&&\pd{^2\lp^\gamma}{v^\mu\partial v^\nu}(x,0,v)=0,\\
\end{aligned}
\end{equation}
which follow from~\eqref{p(x,u,0)=u}. The only relevant second order derivatives are given by
\begin{equation}
\label{local.structure.functions}
C^\gamma_{\mu\nu}(x) \equiv \pd{^2\lp^\gamma}{u^\mu\partial
v^\nu}(x,0,0) -\pd{^2\lp^\gamma}{v^\mu\partial u^\nu}(x,0,0).
\end{equation}
{}From the definition of $L^\gamma_\mu$ and $R^\gamma_\mu$ it
follows that
\begin{equation}
\begin{aligned}
C^\gamma_{\mu\nu}(x)
&=\pd{L^\gamma_\nu}{u^\mu}(x,0)-\pd{L^\gamma_\mu}{u^\nu}(x,0)\\
&=\pd{R^\gamma_\mu}{v^\nu}(x,0)-\pd{R^\gamma_\nu}{v^\mu}(x,0).
\end{aligned}
\end{equation}
On the other hand, if $i: G \to G$ is the inversion we have that
\[
i(x, u) = (\lbeta(x,u), \iota(x, u))
\]
and the condition $i(\epsilon(x)) = \epsilon(x)$, for all $x\in
M$, implies that
\[
\iota(x, 0) = 0.
\]
Moreover, using that $\lp(x, u, \iota(x, u)) = 0$, we deduce that
\[
\displaystyle \frac{\partial \lp^{\gamma}}{\partial u^{\mu}}(x, u,
\iota(x, u)) + \frac{\partial \iota^{\nu}}{\partial u^{\mu}}(x, u)
\frac{\partial \lp^{\gamma}}{\partial v^{\nu}}(x, u, \iota(x,u)) =
0.
\]
Thus, from (\ref{pd and pd2 of p}), we obtain that
\begin{equation}
\displaystyle \frac{\partial \iota^{\gamma}}{\partial u^{\mu}}(x,
0) = -\delta_{\mu}^{\gamma}.
\end{equation}

\subsection{Invariant vector fields}
The local expression for left-{}  and right-trans\-la\-tions are easy
to obtain. For $g_0\in \calw\subset G$ the left translation
$l_{g_0}$ is the map
$\map{l_{g_0}}{\alpha^{-1}(\beta(g_0))}{\alpha^{-1}(\alpha(g_0))}$,
given by $l_{g_0}g=g_0g$. If $g_0$ has coordinates $(x_0,u_0)$,
then the elements on the $\alpha$-fiber  $\alpha^{-1}(\beta(g_0))$
have coordinates of the form $(\lbeta(x_0,u_0),v)$, and the
coordinates of $l_{g_0}g$ are $(x_0,\lp(x_0,u_0,v))$. We will
write
\begin{equation}
l_{(x_0,u_0)}(\lbeta(x_0,u_0),v)=(x_0,\lp(x_0,u_0,v)).
\end{equation}
Similarly, for $h_0\in \calw\subset G$ the right translation
map $\map{r_{h_0}}{\beta^{-1}(\alpha(h_0))}{\beta^{-1}(\beta(h_0))}$, is defined
by $r_{h_0}g=gh_0$. If $h_0$ has coordinates $(x_0,u_0)$,
then the elements on the $\beta$-fiber  $\beta^{-1}(\alpha(h_0))$
have coordinates of the form $(x,u)$ with the restriction
$\lbeta(x,u)=x_0$, and the coordinates of $r_{h_0}g$ are
$(x,\lp(x,u,u_0))$. We will write
\begin{equation}
r_{(x_0,u_0)}(x,u)=(x,\lp(x,u,u_0)).
\end{equation}

A left-invariant vector field is of the form
$\lvec{X}(g)=T_{\epsilon(\beta(g))}l_g(v)$ for $v\in \ker
T_{\epsilon(\beta(g))}\alpha$. To obtain a local basis of
left-invariant vector fields we can take the local coordinate
basis $e_\gamma=\pd{}{u^\gamma}|_{\epsilon(\beta(g))}$ of $\ker
T_{\epsilon(\beta(g))}\alpha$. Thus, for $g\in G$ with coordinates
$(x,u)$, we have
\begin{equation}
\label{left-alpha}
\lvec{e_{\gamma}}(g)=T_{\epsilon(\beta(g))}l_g\left(\pd{}{u^\gamma}\Big|_{\epsilon(\beta(g))}\right)
=\pd{\lp^\mu}{v^\gamma}(x,u,0)\pd{}{u^\mu}\Big|_g
=L^\mu_\gamma(x,u)\pd{}{u^\mu}\Big|_{(x,u)}.
\end{equation}

Similarly, a right-invariant vector field can be written in the form
$\rvec{X}(g) = T_{\epsilon(\alpha(g))}r_g(v)$ for $v\in \ker
T_{\epsilon(\alpha(g))}\beta$. To obtain a local basis of
right-invariant vector fields we first have to look for a basis of
the vector space $\ker T_{\epsilon(\alpha(g))}\beta$. From the
definition of the functions $\rho^i_\gamma$, it follows easily
that the vectors
$f_\gamma=\rho^i_\gamma\pd{}{x^i}-\pd{}{u^\gamma}$ are in $\ker
T_{\epsilon(\alpha(g))}\beta$, and moreover they are related to
the vectors $e_\gamma$ by the inversion map, that is
$Ti(e_\gamma)=f_\gamma$.
%
It follows that a basis of right invariant vector fields is given by
\begin{equation}
\label{right-alpha}
\begin{aligned}
\rvec{e_{\gamma}}(g) &=T_{\epsilon(\alpha(g))}r_g\left(
   -\rho^i_\gamma\pd{}{x^i}\Big|_{\epsilon(\alpha(g))}
   +\pd{}{u^\gamma}\Big|_{\epsilon(\alpha(g))}\right)\\
&=-\rho^i_\gamma(x)\pd{}{x^i}\Big|_{g}
 +\left( -\rho^i_\gamma(x)\pd{\lp^\mu}{x^i}(x,0,u)
        +\pd{\lp^\mu}{u^\gamma}(x,0,u)
  \right)\pd{}{u^\mu}\Big|_{g}
  \\
&=-\rho^i_\gamma(x)\pd{}{x^i}\Big|_{g}
 +R^\mu_\gamma(x,u)\pd{}{u^\mu}\Big|_{g},
\end{aligned}
\end{equation}
where as before $(x,u)$ are the coordinates for $g\in G$. Note that, from~\eqref{p(x,u,0)=u}, we deduce that $\displaystyle\pd{\lp^\mu}{x^i}(x,0,u)=0$.

\subsection{The Lie algebroid of $G$}
The Lie algebroid of $G$ is defined on the vector bundle
$\map{\tau}{E}{M}$ with fiber at the point $x\in M$ given by
$E_x=\ker T_{\epsilon(x)}\alpha$. A local basis of sections of $E$
is given by the coordinate vector fields
$e_\gamma(x)=\pd{}{u^\gamma}|_{\epsilon(x)}$. The anchor is the
map $\map{\rho}{E}{TM}$ defined by
$\rho(a)=T_{\epsilon(x)}\beta(a)$, where $x=\tau(a)$. In local
coordinates, if $a=y^\gamma e_\gamma(x)$ then
$\rho(a)=\rho^i_\gamma(x)y^\gamma\pd{}{x^i}\Big|_x$. The bracket
is defined in terms of the bracket of left-invariant vector
fields. A simple calculation shows that
$[\lvec{e_\gamma},\lvec{e_{\mu}}]=C^\nu_{\gamma\mu}\lvec{e}_\nu$, with $C^{\nu}_{\gamma\mu}$ given by~\eqref{local.structure.functions},  from where we get $[e_\gamma,e_\mu]=C^\nu_{\gamma\mu}e_\nu$.

\subsection{Discrete Euler-Lagrange equations}
\label{Implicit-theorem}
 Consider now a discrete Lagrangian
function $L_d$. A composable pair $(g,h)\in G_2$ satisfies the
Euler-Lagrange equations for $L_d$ if
\begin{equation}
\label{Euler-Lagrange}
\lvec{X}(g)(L_d)=\rvec{X}(h)(L_d)\qquad\text{for every section $X$ of $E$.}
\end{equation}
If both $g$ and $h$ are on the same symmetric neighborhood
$\calw$, with coordinates $(x,u)$ for $g$ and $(y,v)$ for $h$, we
can apply the local results above and we readily get the
coordinate expression of the Euler-Lagrange equations
\begin{gather}
L^\mu_\gamma(x,u)\pd{L_d}{u^\mu}(x,u)+\rho^i_\gamma(y)\pd{L_d}{x^i}(y,v)-
R^\mu_\gamma(y,v)\pd{L_d}{u^\mu}(y,v)=0
\label{Euler-Lagrange1}\\
y=\lbeta(x,u),
\label{Euler-Lagrange2}
\end{gather}
where the second equation takes into account that $\beta(g)=\alpha(h)$.

Assume that we have a solution $(g_0,h_0)\in G_2$ of the
Euler-Lagrange equations. To analyze the existence of solution of
the Euler-Lagrange equations for elements $(g,h)\in G_2$ near
$(g_0,h_0)$, we can apply the implicit function theorem. In the
application of such a theorem, the relevant matrix
$[(\mathbb{F}L_d)^{\gamma}_{\mu}]_{\gamma,\mu=1,\ldots,m}$ is the
following
\begin{multline}
(\mathbb{F}L_d)^{\gamma}_{\mu}(x,u) =
\rho^i_\mu(x)\pd{^2L_d}{x^i\partial
u^\gamma}(x,u)
-\pd{R^\nu_\mu}{u^\gamma}(x,u)\pd{L_d}{u^\nu}(x,u)+\\
-R^\nu_\mu(x,u)\frac{\partial^2L_d}{\partial u^\nu
\partial u^\gamma}(x,u).
\end{multline}

\begin{proposition}
Let $(y_0,v_0)$ be the coordinates of the point $h_0$. The following statements are equivalent.
\begin{itemize}
\item The matrix $(\mathbb{F}L_d)^{\alpha}_{\beta}(y_0,v_0)$ is regular.
\item The Poincare-Cartan 2-section $\Omega_{L_d}$ is non-degenerate at the point $h_0$.
\item The map $\mathbb{F}^-{L_d}$ is a local diffeomorphism at $h_0$.
\end{itemize}
Any of them implies the following: There exist open neighborhoods
${\mathcal X}_{0} \subseteq {\mathcal W}$ and ${\mathcal Y}_{0}
\subseteq {\mathcal W}$ of $g_{0}$ and $h_{0}$ such that if $g\in
{\mathcal X}_{0}$ then there is a unique $\Psi(g) = h \in
{\mathcal Y}_{0}$ satisfying that the pair $(g, h)$ is a
solution of the Euler-Lagrange equations for $L_{d}$. In fact, the
map $\Psi: {\mathcal X}_{0} \to {\mathcal Y}_{0}$ is a local
discrete Euler-Lagrange evolution operator.
\end{proposition}
\begin{proof}
We first notice that the local expression of the map $\mathbb{F}^-_{L_d}$ is
\[
(\mathbb{F}^-_{L_d})(x, u) = \left(x,
-\rho^{i}_{\gamma}(x)\pd{L_d}{x^i}(x, u) + R^\mu_\gamma(x,
u)\pd{L_d}{u^\mu}(x, u)\right).
\]
The differential of this local function at the point $h\equiv(y_0,v_0)$ is of the form
\[
\begin{bmatrix}
I_n&0\\{*}&-\mathbb{F}L_d(y_0,v_0)
\end{bmatrix}
\]
from where the equivalence of the first and the third assertions immediately follows.

For the second we just take a local basis of sections
$\{e_\gamma\}$ of $E$, defined in a neighborhood of $x_0$ and
associated to the $\alpha$-vertical vector fields
$\pd{}{u^\gamma}$, and we compute the value of $\Omega_{L_d}$ on
the associated basis $\{ e_\gamma^{(1,0)},e_\gamma^{(0,1)}\}$.
From \eqref{Omega1} and \eqref{Omega2} it follows that
$\Omega_{L_d}(h_0)$ is regular if and only if the matrix
\[
\Omega_{\gamma\mu}(y_0,v_0)=\Omega_{L_d}(e_\gamma^{(1,0)},e_\mu^{(0,1)})(h_0)
=\lvec{e_\mu}(\rvec{e_\gamma}(L_d))(h_0)
\]
is regular. From the expressions~\eqref{left-alpha} and~\eqref{right-alpha} this matrix is
\begin{align*}
\Omega_{\gamma\mu}(y_0,v_0) &=-L^\theta_\mu\pd{}{u^\theta}\left(
 \rho^i_\gamma\pd{L_d}{x^i}-R^\nu_\gamma\pd{L_d}{u^\nu}
\right)(y_0,v_0)\\
&=-L^\theta_\mu\left(
 \rho^i_\gamma\pd{^2L_d}{x^i\partial u^\theta}
 -\pd{R^\sigma_\gamma}{u^\theta}\pd{L_d}{u^\sigma}
 -R^\sigma_\gamma\pd{^2L_d}{u^\sigma\partial u^\theta}
\right)(y_0,v_0)\\
&=-L^\theta_\mu(y_0,v_0)(\mathbb{F}L_d)_\gamma^\theta(y_0,v_0).
\end{align*}
It follows from~\eqref{left-alpha} that the matrix $L^\theta_\mu(y_{0},v_{0})$ is regular. Thus, we get that the regularity of the matrix
$(\mathbb{F}L_d)^{\theta}_{\gamma}(y_0,v_0)$ is equivalent to the
regularity of $\Omega_{L_d}(h_0)$.

Finally, let $\lambda_\mu$ be the left-hand side of the
discrete Euler-Lagrange equations~\eqref{Euler-Lagrange1} once the
equation~\eqref{Euler-Lagrange2} has been used
\begin{multline}
\lambda_\mu(x,u,v)=
L^\gamma_\mu(x,u)\pd{L_d}{u^\gamma}(x,u)+\rho^i_\mu(\lbeta(x,u))\pd{L_d}{x^i}(\lbeta(x,u),v)+\\
-R^\gamma_\mu(\lbeta(x,u),v)\pd{L_d}{u^\gamma}(\lbeta(x,u),v).
\end{multline}
We want to study the existence of solution of the discrete Euler-Lagrange equations $\lambda_\mu(x,u,v)=0$ in a
neighborhood of the point $(x_0,u_0,v_0)$, where $(x,u)$ are the
data and $v$ is the unknown. Applying the implicit function
theorem we have to study the regularity of the matrix
$\pd{\lambda_\mu}{v^\gamma}(x_0,u_0,v_0)$. A straightforward calculation shows that this matrix
is just
\[
\pd{\lambda_\mu}{v^\gamma}(x_0,u_0,v_0)=(\mathbb{F}L_d)^{\gamma}_{\mu}(y_0,v_0)
\qquad\text{with $y_0=\lbeta(x_0,u_0)$,}
\]
from where our last assertion readily follows.
\end{proof}

\begin{remark}{\rm
Suppose that $h_{0} \in \epsilon(M)$. Then, $h_{0}$ has local coordinates $(y_{0}, 0)$. Moreover, from (\ref{rho-L-R}) and (\ref{pd and pd2 of p}), it follows that the matrix  $L_{\mu}^{\gamma}(y_{0},0)$ is the identity matrix. Therefore, we deduce that the matrix $\Omega_{\gamma\mu}(y_{0},0)$ is, up to the sign, the matrix $(\mathbb{F}L_{d})^{\mu}_{\gamma}(y_{0},0)$ (see the proof of
Theorem \ref{DELE-operator1} in Section \ref{GdE-Leo}).}
\end{remark}

In general the Euler-Lagrange equations are understood as the
equations determining $h\equiv(y,v)$ from the already known data
$g\equiv(x,u)$, as we did in the proof of the above theorem. However, we can also try to
solve these equations backwards to obtain $(x,u)$ from $(y,v)$.
Instead of applying again the implicit function theorem to the system of
equations~\eqref{Euler-Lagrange1} and \eqref{Euler-Lagrange2} we can rewrite the equations in a different coordinate system, adapted to the fibration $\beta$, as
follows.  On the open subset $\bar{\calu}=i(\calu)$ we consider
the local coordinates $(\bar{x},\bar{u})$ defined as
$(\bar{x},\bar{u})=(x,u)\circ i$, where $i$ is the inversion map
in  the groupoid. Since $\beta\circ i=\alpha$, we have that these
new coordinates are adapted to the submersion $\beta$ and we can
proceed as in the previous case. In these coordinates, a basis of the left-invariant and right-invariant vector fields is
\begin{gather}
\label{inverse-right-alpha}
\lvec{e_{\gamma}}(\bar{x},\bar{u})
=-\rho^i_\gamma(\bar{x})\pd{}{\bar{x}^i}\Big|_{(\bar{x},\bar{u})}
 +R^\mu_\gamma(\bar{x},\bar{u})\pd{}{\bar{u}^\mu}\Big|_{(\bar{x},\bar{u})}\,\,,
\\
\label{inverse-left-alpha}
\rvec{e_{\gamma}}(\bar{x},\bar{u})
=L^\mu_\gamma(\bar{x},\bar{u})\pd{}{\bar{u}^\mu}\Big|_{(\bar{x},\bar{u})}\,.
\end{gather}

\subsection{Examples}
We present here some illustrative examples.

\subsubsection{Pair or Banal groupoid}\label{ex-banal}

We consider as a first example the pair (banal) groupoid $G=M\times M$, where the
structural maps are
\begin{eqnarray*}
&\alpha(x, y)=x, \; \; \beta(x, y)=y, \; \; \epsilon(x)=(x,x), \; \; i(x, y)=(y, x),&\\
&m((x, y), (y, z))=(x, z).&
\end{eqnarray*}
The Lie algebroid of $G$ is isomorphic to the
standard Lie algebroid $\tau_{M}: TM \to M$, therefore the pair groupoid is considered as the discrete phase space for discretization of Lagrangian functions $L: TM\to \R$.

 Let $x_0\in M$ and a local coordinate system $(x^i)$ defined on a neighborhood  ${\mathcal V}'$ of $x_0$.  Then ${\mathcal U}={\mathcal V}'\times {\mathcal V}'$   is obviously a symmetric neighborhood. For a fixed $h>0$, an associated coordinate system adapted to the $\alpha$-projection is
 \[
 x^i(x_1,x_2)=x^i(x_1) \qquad \hbox{and} \qquad u^i(x_1,x_2)= \frac{x^i(x_2)-x^i(x_1)}{h}.
 \]
Obviously the identities correspond to the elements $(x^i,0)$ in this coordinate system and
 $\alpha(x, u)=x$ and $\beta(x, u)=x+hu$.  On the other hand, the composition of two elements $(x, u)$ and $(x+hu, v)$ is
 $(x, u+v)$. Therefore $\lp(x,u,v)=u+v$.
  The inversion map is now $i(x,u)=(x+hu, -u)$.

 If we take the natural local coordinate basis  $\left\{\frac{\partial}{\partial u^i}\large|_{\epsilon(x_1)} \right\}$ of $\ker T_{ \epsilon(x_1)}\alpha$ then
\begin{align*}
 \lvec{\frac{\partial}{\partial u^i}}
   &=\frac{\partial}{\partial u^i} \\
 \rvec{\frac{\partial}{\partial u^i}}
   &=-h\frac{\partial}{\partial x^i}+\frac{\partial}{\partial u^i}.
\end{align*}
Therefore, the discrete Euler-Lagrange equations are
\begin{gather*}
 \frac{\partial L_d}{\partial u^i}(x,u)+h\frac{\partial L_d}{\partial x^i}(y, v)
   -\frac{\partial L_d}{\partial u^i}(y, v)=0\\
 y=x+hu.
\end{gather*}
The discrete Lagrangian is regular if the matrix
\[
h\frac{\partial^2 L_d}{\partial x^i\partial u^j}(x,u)-\frac{\partial^2 L_d}{\partial u^i\partial u^j}(x,u)
\]
is non singular.

 \begin{example}
 {\rm
 As a concrete simple example, consider the continuous Lagrangian  $L\colon \R^{2n}\to \R$:
\[
L({x}, \dot{x})=\frac{1}{2} \dot{x}^T M \dot{x} -V(x)
\]
(with $M$ a  constant symmetric invertible matrix). A typical discretization for the lagrangian is, for instance,
\[
L_d(x, u)=\frac{h}{2} {u}^T M {u} -hV(x+\frac{h}{2}u).
\]
Then, the discrete Euler-Lagrange equations are:
\begin{gather*}
M\frac{v-u}{h}= -\frac{1}{2}\left(\frac{\partial V}{\partial x}(x+\frac{h}{2}u)+\frac{\partial V}{\partial x}(y+\frac{h}{2}v)\right)
\\
y=x+hu,
\end{gather*}
which leads us the classical  implicit midpoint rule.

If on the other hand, we take the discretization
\begin{eqnarray*}
L_d(x, u)&=&\frac{h}{2}\left(L(x, u)+L(x+hu, u)\right)\\
&=&\frac{h}{2} {u}^T M {u} -\frac{h}{2}\left(V(x)+V(x+hu)\right)\, ,
\end{eqnarray*}
then the corresponding discrete Euler-Lagrange equations are: 
\begin{gather*}
M\frac{v-u}{h}= -\frac{\partial V}{\partial x}(y)
\\
y=x+hu,
\end{gather*}
which is a representation of the St\"ormer-Verlet method.


}

\end{example}

\subsubsection{Lie groups}\label{ex-lie}

Another interesting example corresponds to the case of Lie groups. In this case, 
we consider a Lie group $G$ as a groupoid  over one point
$M=\{\frak e \}$, the identity element of $G$. The structural maps
are
\[
\alpha(g)={\frak e} , \; \; \beta(g)={\frak e} , \; \;
\epsilon({\frak e})={\frak e}, \; \; i(g)=g^{-1}, \; \; m(g,
h)=gh, \; \; \mbox{ for } g, h \in G.
\]
The Lie algebroid associated with $G$ is just the Lie algebra
${\frak g}=T_{\frak e}G$ of $G$.

Near of the identity, it is interesting to regard the elements $g\in G$ as small displacements on
the Lie group. Thus, it is possible to express each term through a Lie
algebra element that can be regarded as the averaged velocity of this
displacement. This is typically accomplished using a retraction
  map $\tau: {\mathfrak g}\to G$ which is an analytic local
diffeomorphism around the identity such that $\tau(\xi)\tau(-\xi)={\frak e}$,
where $\xi\in\mathfrak g$ (see \cite{Rabee}). Thereby $\tau$ provides a local chart on the Lie
group.

Given a  retraction map $\tau$ its right trivialized tangent $\mbox{d}\tau_{\xi}:\mathfrak{g}\rightarrow\mathfrak{g}$ and is inverse $\mbox{d}\tau_{\xi}^{-1}:\mathfrak{g}\rightarrow\mathfrak{g}$ are defined for all   $\eta\in\mathfrak{g}$ as follows
  \begin{eqnarray*}
    T_{\xi}\tau (\eta)&=&T_{\frak e} r_{\tau(\xi)}\left(\mbox{d}\tau_{\xi}(\eta)\right)\equiv \mbox{d}\tau_{\xi}(\eta)\,\tau(h\xi),\\
    T_{\tau(\xi)}\tau^{-1}((T_{\frak e} r_{\tau(\xi)})\eta)&=&\mbox{d}\tau^{-1}_{\xi}(\eta).
  \end{eqnarray*}

The retraction map allows us to transport locally the Lie group structure on an open symmetric neighborhood of $G$ to a local $\tau$-dependent Lie group structure on the Lie algebra ${\mathfrak g}$ defined on a local neighborhood $V$ of $0\in {\mathfrak g}$.
We will write $\tau(h\xi)=g$ for an enough small  time step $h>0$ such that $h\xi\in V\subseteq {\mathfrak g} $.

Now, on the Lie algebra it is easy to consider local coordinates because it is a vector space. In consequence,  fixing a basis $\{e_{\gamma}\}$ of ${\mathfrak g}$, we induce  coordinates $(u^{\gamma})$ on ${\mathfrak g}$. 

In these coordinates, a basis of left- and right-invariant vector fields is 
\begin{eqnarray*}
\lvec{e_{\gamma}}(\eta)&=&T_{\tau(h\eta)}\tau^{-1}(\tau(h\eta)e_{\gamma})=
\mbox{d}\tau^{-1}_{h\eta} ({\rm Ad}_{\tau(h\eta)}e_{\gamma})\\
\rvec{e_{\gamma}}(\eta)&=&T_{\tau(h\eta)}\tau^{-1}(e_{\gamma}\tau(h\eta))=
\mbox{d}\tau^{-1}_{h\eta}(e_{\gamma}),
\end{eqnarray*}
where $\eta\in {\mathfrak g}$.
Given a lagrangian $l: {\mathfrak g}\rightarrow \R$, we deduce that the   discrete Euler-Poincar\'e equations are: 
\[
\lvec{e_{\gamma}}(\eta_k)(l)-\rvec{e_{\gamma}}(\eta_{k+1})(l)=0,
\]
or, alternatively, as discrete Lie-Poisson equations:
\[
Ad^*_{\tau({h\eta_k})}\mu_k-\mu_{k+1}=0,
\]
where
 \[
\mu_k=(\mbox{d}\tau^{-1}_{h\eta_k}) ^*\frac{\partial l}{\partial \xi}(\eta_k)\; .
\]
\begin{example}
{\rm 
As an example of retraction map  we will consider the Cayley transform which is one of the most computationally efficient  parametrizations of a Lie group. The Cayley map ${\rm cay}:{\mathfrak g}\rightarrow  G$ is defined by ${\rm cay}(\xi)=({\frak e}-\frac{\xi}{2})^{-1}({\frak e}+\frac{\xi}{2})$ and is valid for a general class of quadratic groups as for instance $SO(n)$, $SE(n)$ or $Sp(n)$  (see \cite{HaLuWa}).  Its right trivialized derivative and inverse are given by
\begin{eqnarray*}
\mbox{d}{\rm cay}_{\xi}\,\eta&=&({\frak e}-\frac{\xi}{2})^{-1}\,\eta\,({\frak e}+\frac{\xi}{2})^{-1}\\
\mbox{d}{\rm cay}_{\xi}^{-1}\,\eta&=&({\frak e}-\frac{\xi}{2})\,\eta\,({\frak e}+\frac{\xi}{2}).
\end{eqnarray*}

In this case it is possible to write more explicitly the Lie-Poincar\'e equations. In fact, we have that  
\begin{eqnarray*}
\lvec{e_{\gamma}}(\eta)&=& ({\frak e}+\frac{h\eta}{2})\,e_{\gamma}\,({\frak e}-\frac{h\eta}{2})\\
\rvec{e_{\gamma}}(\eta)&=&({\frak e}-\frac{h\eta}{2})\,e_{\gamma}\,({\frak e}+\frac{h\eta}{2}),
\end{eqnarray*}
with $\eta \in \mathfrak{so}(3)$.
Therefore, the discrete Euler-Poincar\'e equations are: 
\begin{eqnarray*}
0&=&\langle ({\frak e}+\frac{h\eta_k}{2})\,e_{\gamma}\,({\frak e}-\frac{h\eta_k}{2}), \frac{\partial l}{\partial \xi}(\eta_k)\rangle\\
&&-
\langle ({\frak e}-\frac{h\eta_{k+1}}{2})\,e_{\gamma}\,({\frak e}+\frac{h\eta_{k+1}}{2}),\frac{\partial l}{\partial \xi}(\eta_{k+1})\rangle \\
&=&\langle e_{\gamma}-\frac{h}{2}[e_{\gamma}, \eta_k]-\frac{h^2}{4}\eta_ke_{\gamma}\eta_k, 
\frac{\partial l}{\partial \xi}(\eta_{k})\rangle\\
&&
-\langle e_{\gamma}+\frac{h}{2}[e_{\gamma}, \eta_{k+1}]-\frac{h^2}{4}\eta_{k+1}e_{\gamma}\eta_{k+1}, 
\frac{\partial l}{\partial \xi}(\eta_{k+1})\rangle.
\end{eqnarray*}

As a concrete example, we consider the case of the group $SO(3)$, using the Cayley transformation and fixing the  standard basis of its Lie algebra $\mathfrak{so}(3)$: 
\[
e_1=\left(
\begin{array}{rrr}
0&0&0\\
0&0&-1\\
0&1&0
\end{array}
\right)
,
\quad 
e_2=\left(
\begin{array}{rrr}
0&0&1\\
0&0&0\\
-1&0&0
\end{array}
\right) ,
\quad 
e_3=\left(
\begin{array}{rrr}
0&-1&0\\
1&0&0\\
0&0&0
\end{array}
\right)
 .
\] 
Thus, the elements $\eta_k$, $\eta_{k+1}$ in ${\mathfrak g}$ will have coordinates 
$(x_k, y_k, z_k)$ and $(x_{k+1}, y_{k+1}, z_{k+1})$, respectively, in this fixed basis. 
This allows us to write the previous discrete Euler-Poincar\'e equations as follows: 
\begin{eqnarray*}
0&=&\left(1+\frac{h^2x_k^2}{4}\right)\left.\frac{\partial L}{\partial x}\right|_k
+\left(\frac{h^2x_ky_k}{4}+\frac{hz_k}{2}\right)\left.\frac{\partial L}{\partial y}\right|_k\\
&&-\left(\frac{hy_k}{2}-\frac{h^2x_kz_k}{4}\right)\left.\frac{\partial L}{\partial z}\right|_k
-\left(1+\frac{h^2x_{k+1}^2}{4}\right)\left.\frac{\partial L}{\partial x}\right|_{k+1}\\
&&+\left(\frac{hz_{k+1}}{2}-\frac{h^2x_{k+1}y_{k+1}}{4}\right)\left.\frac{\partial L}{\partial y}\right|_{k+1}
-\left(\frac{hy_{k+1}}{2}+\frac{h^2x_{k+1}z_{k+1}}{4}\right)\left.\frac{\partial L}{\partial z}\right|_{k+1}\\
0&=&
\left(\frac{h^2x_ky_k}{4}-\frac{hz_k}{2}+\right)\left.\frac{\partial L}{\partial x}\right|_k
+\left(1+\frac{h^2y_k^2}{4}\right)\left.\frac{\partial L}{\partial y}\right|_k\\
&&+\left(\frac{hx_k}{2}+\frac{h^2y_kz_k}{4}\right)\left.\frac{\partial L}{\partial z}\right|_k
-\left(\frac{hz_{k+1}}{2}+\frac{h^2x_{k+1}y_{k+1}}{4}\right)\left.\frac{\partial L}{\partial x}\right|_{k+1}\\
&&-\left(1+\frac{h^2y_{k+1}^2}{4}\right)\left.\frac{\partial L}{\partial y}\right|_{k+1}
+\left(\frac{hx_{k+1}}{2}-\frac{h^2y_{k+1}z_{k+1}}{4}\right)\left.\frac{\partial L}{\partial z}\right|_{k+1}\\
0&=&
\left(\frac{hy_k}{2}+\frac{h^2x_kz_k}{4}\right)\left.\frac{\partial L}{\partial x}\right|_k
-\left(\frac{hx_k}{2}-\frac{h^2y_kz_k}{4}\right)\left.\frac{\partial L}{\partial y}\right|_k\\
&&+\left(1+\frac{h^2z_{k}}{4}\right)\left.\frac{\partial L}{\partial z}\right|_k
+\left(\frac{hy_{k+1}}{2}-\frac{h^2x_{k+1}z_{k+1}}{4}\right)\left.\frac{\partial L}{\partial x}\right|_{k+1}\\
&&-\left(\frac{hx_{k+1}}{2}+\frac{h^2y_{k+1}z_{k+1}}{4}\right)\left.\frac{\partial L}{\partial y}\right|_{k+1}
-\left(1+\frac{h^2z_{k+1}}{4}\right)\left.\frac{\partial L}{\partial z}\right|_{k+1}\; ,
\end{eqnarray*}
where $L:\R^3\rightarrow \R$ is defined by $L(x,y,z)=l(xe_1+ye_2+ze_3)$.

}

\end{example}

\subsection{Transformation or action Lie groupoid}
\label{TaLg}
The theory of Lie grou\-poids covers other interesting examples that are useful for the construction of variational integrators for different mechanical systems. This is the case of transformation or action Lie groupoids. 
Let $\widetilde{G}$ be a Lie group and $\cdot:M\times\widetilde{G}\to M$, $(x,\tilde{g})\in
M\times \widetilde{G}\mapsto x\tilde{g},$ a right action of $\widetilde{G}$  on $M$.
 Consider the  Lie groupoid $G=M\times \widetilde{G}$ over $M$ with
structural maps given by
\begin{equation}\label{*tilde}
\begin{array}{l}
{\alpha}(x,\tilde{g})=x,\;\;\; {\beta}(x,\tilde{g})=x\tilde{g},\;\;\;
{\epsilon }(x)=(x,\tilde{\frak e}),\\
{m}((x,\tilde{g}),(x\tilde{g},\tilde{g}'))=(x,\tilde{g}\tilde{g}'),\;\;\;
\text{and }\;\;\;
{i}(x,\tilde{g})=(x\tilde{g}, \tilde{g}^{-1}).
\end{array}
\end{equation}
where $\tilde{\frak e}$ is the identity on $\widetilde{G}$. The Lie groupoid $G$ is called the action  or transformation Lie groupoid. Its associated Lie algebroid is the action algebroid 
$pr_1:M\times
{\tilde{\mathfrak{ g}}}\to M$ where 
$\tilde{{\mathfrak g}}$ is the Lie algebra of Lie group $\tilde{G}$ (see \cite{MaMaMa} for  details about the Lie algebroid structure).
We have that $\Gamma(pr_1)\cong \{\bar{\eta}:M\to \tilde{\mathfrak
g}\; |\; \mbox{ where }\bar{\eta} \mbox{ is  smooth mapping}\}$. 
In particular, if $\eta\in \tilde{\mathfrak g}$ then $\eta$ defines a constant
section $C_\eta:M\to \tilde{{\mathfrak g}}$ of $pr_1:M\times \tilde{{\mathfrak
g}}\to M$: $C_{\eta}(x)=(x, \eta)$ for all $x\in M$. 
It is possible to check that the corresponding left- and right-invariant vector fields on $G$ are (see \cite{MaMaMa}):
\begin{equation}\label{*+}
\lvec{C}_\eta(x,\tilde{g})=(0_x,\lvec{\eta}(\tilde{g})),\;\;\;\;
\rvec{C}_\eta(x,\tilde{g})=(-\eta_M(x),\rvec{\eta}(\tilde{g})),
\end{equation}
for $(x,\tilde{g})\in G=M\times \widetilde{G}$ and where $\eta_M$ is the infinitesimal generator of the right action 
$\cdot:M\times \widetilde{G}\to M$ associated with $\eta$. 


Let  $L_d: G=M\times \widetilde{G}\to \R$ be a discrete  Lagrangian. 
Then, a composable
pair $((x_k,\tilde{g}_k),(x_k\tilde{g}_k,\tilde{g}_{k+1}))\in G_2$ is a solution of the
discrete Euler-Lagrange equations for $L_d$ if
\[
\lvec{C}_\eta(x_k,\tilde{g}_k)(L)-\rvec{C}_\eta(x_k\tilde{g}_k,\tilde{g}_{k+1})(L)=0, \mbox{
for all } \eta\in {\mathfrak g}.
\]
Given a retraction map $\tau: \tilde{\mathfrak g}\rightarrow  \tilde{G}$, we can transport 
the Lie group structure on an open symmetric neighborhood of $\tilde{\mathfrak e}$ in $\tilde{G}$ to an open neighborhood of $0$ in $\tilde{\mathfrak g}$ as in Subsection
\ref{ex-lie}. Moreover, if $x_0\in M$, we may assume, without loss of generality, that this neighborhood acts on an open neighborhood $U$ of the point $x_0\in M$.  Thus if $\tilde{\xi}\in {\mathfrak g}$ the flow $\Phi_{\tilde{\xi}_M}$ of the fundamental vector field ${\tilde \xi}_M$ associated with this local action is given by 
\[
\Phi_{\tilde{\xi}_M}(t, x)=x{\mbox exp}_{\tilde{G}}(t(h\tilde{\xi}))
\]
for $t\in \R$ and $x\in U$. Therefore, as in the previous section, 
 fixed  
 a basis $\{e_{\gamma}\}$ of $\tilde{{\mathfrak g}}$, we induce a basis of left- and right-invariant vector fields 
\begin{eqnarray*}
\lvec{C_{e_{\gamma}}}(x, \tilde{\eta})&=&(0,
\mbox{d}\tau^{-1}_{h\tilde{\eta}} ({\rm Ad}_{\tau(h\tilde{\eta})}e_{\gamma}))\\
\rvec{C_{e_{\gamma}}}(x, \tilde{\eta})&=&(-h(e_{\gamma})_M(x), 
\mbox{d}\tau^{-1}_{h\tilde{\eta}}(e_{\gamma})),
\end{eqnarray*}
where $\tilde{\eta}\in \tilde{\mathfrak g}$ and $x\in U$.
Given a lagrangian $l: M\times \tilde{\mathfrak g}\rightarrow \R$, we deduce that the   discrete Euler-Lagrange equations  are: 
\begin{eqnarray*}
0&=&\mbox{d}\tau^{-1}_{h\tilde{\eta}_k} ({\rm Ad}_{\tau(h\tilde{\eta}_k)}e_{\gamma})(l_{x_k})
-\mbox{d}\tau^{-1}_{h\tilde{\eta}_{k+1}}(e_{\gamma})(l_{x_{k+1}})
+h(e_{\gamma})_M(x_{k+1})(l_{\tilde{\eta}_{k+1}}), \\
x_{k+1}&=&x_k\tau(h\tilde{\eta}_k)\; .
\end{eqnarray*}
where  for every $\tilde{\eta}\in \tilde{\mathfrak g}$ (resp., $x\in M$)  we denote by $l_{\tilde{\eta}}$
(resp., $l_x$) the real function on $M$ (resp., on $\tilde{\mathfrak g}$) given by
$l_{\tilde{\eta}}(y)=l(y,{\tilde{\eta}})$ (resp., $l_x(\tilde{\eta}')=l(x,{\tilde{\eta}}')$). 
\begin{example}
{\rm 
As a typical example of a discrete system defined on a transformation Lie groupoid
 consider a discretization of the heavy top \cite{rat,Ma,MaMaMa}.
This system is modelled on the 
transformation Lie algebroid $\tau:S^2\times\mathfrak{so}(3)\to
S^2$ with Lagrangian
\[
L_c(\Gamma,\Omega)=\frac{1}{2}\Omega\cdot
{\mathbb I}\Omega-mgd\,\Gamma\cdot\mathrm{e},
\]
where $\Omega\in\R^3\simeq\mathfrak{so}(3)$ is the angular
velocity, $\Gamma$ is the direction opposite to the gravity and
$\mathrm{e}$ is a unit vector in the direction from the fixed point to the
center of mass, all them expressed in a frame fixed to the body.
The constants $m$, $g$ and $d$ are respectively the mass of the
body, the strength of the gravitational acceleration and the
distance from the fixed point to the center of mass. The matrix
${\mathbb I}$ is the inertia tensor of the body.

We will also use the Cayley transformation on $SO(3)$ to describe the discrete Euler-Lagrange equations for the heavy top. 
We have that  
\begin{eqnarray*}
\lvec{C_{e_{\gamma}}}(\Gamma, \eta)&=&\left(0,  ({{\frak e}}+\frac{h{\eta}}{2})\,e_{\gamma}\,({\frak e}-\frac{h\eta}{2})\right)\\
\rvec{C_{e_{\gamma}}}(\Gamma, {\eta})&=&\left(h\Gamma e_{\gamma} ,({\frak e}-\frac{h\eta}{2})\,e_{\gamma}\,({\frak e}+\frac{h\eta}{2})\right),
\end{eqnarray*}
with 
\[\eta=
\left(\begin{array}{rrr}
0&-\Omega_3&\Omega_2\\
\Omega_3&0&-\Omega_1\\
-\Omega_2&\Omega_1&0
\end{array}
\right)\; ,
\]
$\{e_{\gamma}\}$ the standard basis on $SO(3)$ and $\Gamma\in S^2$.

Therefore, the discrete Euler-Lagrange equations are:

\begin{eqnarray*}
0&=&\langle ({\frak e}+\frac{h\eta_k}{2})\,e_{\gamma}\,({\frak e}-\frac{h\eta_k}{2}), \frac{\partial (L_c)_{\Gamma_k}}{\partial \xi}(\eta_k)\rangle\\
&&-
\langle ({\frak e}-\frac{h\eta_{k+1}}{2})\,e_{\gamma}\,({\frak e}+\frac{h\eta_{k+1}}{2}),\frac{\partial (L_c)_{\Gamma_{k+1}}}{\partial \xi}(\eta_{k+1})\rangle\\
&&+h\Gamma_{k+1} e_{\gamma}\cdot\frac{\partial (L_c)_{\eta_{k+1}}}{\partial \Gamma}(\Gamma_{k+1})\\
\Gamma_{k+1}&=&\Gamma_k\;  \mbox{cay }(h\eta_k),
\end{eqnarray*}
or, in other terms
\begin{eqnarray*}
0&=&I_1(\Omega_1)_k\left(1+\frac{h^2(\Omega_1)_k^2}{4}\right)
+I_2(\Omega_2)_{k}\left(\frac{h^2(\Omega_1)_k(\Omega_2)_{k}}{4}+\frac{h(\Omega_3)_{k}}{2}\right)\\
&&-I_3(\Omega_3)_{k}\left(\frac{h(\Omega_2)_{k}}{2}-\frac{h^2(\Omega_1)_k(\Omega_3)_{k}}{4}\right)
-I_1(\Omega_1)_{k+1}\left(1+\frac{h^2(\Omega_1)_{k+1}^2}{4}\right)\\
&&+I_2(\Omega_2)_{k+1}\left(\frac{h(\Omega_3)_{k+1}}{2}-\frac{h^2(\Omega_1)_{k+1}(\Omega_2)_{k+1}}{4}\right)\\
&&
-I_3(\Omega_3)_{k+1}\left(\frac{h(\Omega_2)_{k+1}}{2}+\frac{h^2(\Omega_1)_{k+1}(\Omega_3)_{k+1}}{4}\right)\\
&&+hmgd( Z_{k+1}{\rm e}_2-Y_{k+1}{\rm e}_3)
\end{eqnarray*}
\begin{eqnarray*}
0&=&
I_1(\Omega_1)_k\left(\frac{h^2(\Omega_1)_k(\Omega_2)_{k}}{4}-\frac{h(\Omega_3)_{k}}{2}+\right)
+I_2(\Omega_2)_{k}\left(1+\frac{h^2(\Omega_2)_{k}^2}{4}\right)\\
&&+I_3(\Omega_3)_{k}\left(\frac{h(\Omega_1)_k}{2}+\frac{h^2(\Omega_2)_{k}(\Omega_3)_{k}}{4}\right)
-I_1(\Omega_1)_{k+1}\left(\frac{h(\Omega_3)_{k+1}}{2}+\frac{h^2(\Omega_1)_{k+1}(\Omega_2)_{k+1}}{4}\right)\\
&&-I_2 (\Omega_2)_{k+1}\left(1+\frac{h^2(\Omega_2)_{k+1}^2}{4}\right)
+I_3(\Omega_3)_{k+1}\left(\frac{h(\Omega_1)_{k+1}}{2}-\frac{h^2(\Omega_2)_{k+1}(\Omega_3)_{k+1}}{4}\right)\\
&&+hmgd( X_{k+1}{\rm e}_3-Z_{k+1}{\rm e}_1)
\end{eqnarray*}
\begin{eqnarray*}
0&=&
I_1(\Omega_1)_k\left(\frac{h(\Omega_2)_{k}}{2}+\frac{h^2(\Omega_1)_k(\Omega_3)_{k}}{4}\right)
-I_2(\Omega_2)_{k}\left(\frac{h(\Omega_1)_k}{2}-\frac{h^2(\Omega_2)_{k}(\Omega_3)_{k}}{4}\right)\\
&&+I_3(\Omega_3)_{k}\left(1+\frac{h^2(\Omega_3)_{k}}{4}\right)
+I_1(\Omega_1)_{k+1}\left(\frac{h(\Omega_2)_{k+1}}{2}-\frac{h^2(\Omega_1)_{k+1}(\Omega_3)_{k+1}}{4}\right)\\
&&-I_2(\Omega_2)_{k+1}\left(\frac{h(\Omega_1)_{k+1}}{2}+\frac{h^2(\Omega_2)_{k+1}(\Omega_3)_{k+1}}{4}\right)
-I_3(\Omega_3)_{k+1}\left(1+\frac{h^2(\Omega_3)_{k+1}}{4}\right)\\
&&-hmgd( X_{k+1}{\rm e}_2-Y_{k+1}{\rm e}_1)
\\
0&=&(X_{k+1}, Y_{k+1}, Z_{k+1})-(X_k, Y_k, Z_k)({\frak e}-\frac{h\eta_k}{2})^{-1}({\frak e}+\frac{h\eta_k}{2})
\end{eqnarray*}
where 
$\Gamma_k=(X_k, Y_k, Z_k)\in \R^3$ with $X_k^2+Y_k^2+Z_k^2=1$, ${\rm e}=({\rm e}_1, {\rm e}_2, {\rm e}_3)$ and
\[
\eta_k=
\left(\begin{array}{rrr}
0&-(\Omega_3)_k&(\Omega_2)_k\\
(\Omega_3)_k&0&-(\Omega_1)_k\\
-(\Omega_2)_k&(\Omega_1)_k&0
\end{array}
\right)\; ,
\]
}
\end{example}

\begin{remark}
{\rm 
Our approach also admits other interesting examples. For instance, assume that we have a  
discrete system modeled by a  $L_d: M\times M\times G\to \R$ which is  an approximation of a continuous Lagrangian $L: TM\times {\mathfrak g}\to \R$. This lagrangian typically appears as reduction of a $G$-invariant Lagrangian function $\tilde{L}: T(M\times G)\to \R$ (see \cite{MaMaMa}, for the general case). Of course we can combine the techniques exposed in Subsections \ref{ex-banal} and \ref{ex-lie} to obtain a local description of the corresponding discrete Euler-Lagrange equations in terms of the continuous Lagrangian $L$.
}
\end{remark}

\section{Bisections and discrete Euler-Lagrange evolution operators}
\label{sec:EL-general}

One of the main limitations of the techniques exposed in Section \ref{section:symmetric} is that we  need to work in a neighborhood of the identities of the Lie groupoid $G$. 
Using an enough small time stepping we can guarantee that the evolution of the evolution operator for a discrete Lagrangian takes values on the chosen symmetric neighborhood, even it is possible to adapt the time stepping to make it happen.  Another possibility is to use the   notion of bisections on Lie groupoids. As we will see it will allow us to consider points far from the identities completing our local description of discrete Mechanics.

We consider now the general case of a
solution $(g_0,h_0)\in G_2$ of the Euler-Lagrange equations where
the points $g_0$ and $h_0\in G$ are not necessarily close enough to be contained in a common  symmetric neighborhood. If we want to obtain a local expression of the
discrete Euler-Lagrange operator which connects the above points,
we must choose suitable neighborhoods of $g_0$ and $h_0$. For this
purpose, we will consider a symmetric neighborhood $\calw$, a local bisection through the point $g_0$ and a local bisection through $h_0$. By left-translation and right-translation (induced by these sections) of $\calw$ we will get such neighborhoods.


\subsection{Bisections of a Lie groupoid}
The results contained in this section are well-known in the
literature (see, for instance, \cite{CaWe,Mac}). However, to make
the paper more self-contained, we will include the proofs of them.

Let $G \rightrightarrows M$ be a Lie groupoid with source
$\map{\alpha}{G}{M}$  and target $\map{\beta}{G}{M}$.

\begin{definition}
A \emph{bisection} of $G$ is a closed embedded submanifold
$\Sigma$ of $G$ such that the restrictions of both $\alpha$ and
$\beta$ to $\Sigma$ are diffeomorphisms.
\end{definition}

A bisection defines both a section $\Sigma_\alpha$ of $\alpha$ and
a section $\Sigma_\beta$ of $\beta$ as follows.

\begin{proposition}
Given a bisection $\Sigma$ the map
$\Sigma_\alpha=(\alpha|_\Sigma)^{-1}$ is a smooth section of
$\alpha$ such that $\Im(\Sigma_\alpha)=\Sigma$ and
$\beta\circ\Sigma_\alpha$ is a diffeomorphism. Alternatively, the
map $\Sigma_\beta=(\beta|_\Sigma)^{-1}$ is a smooth section of
$\beta$ such that $\Im(\Sigma_\beta)=\Sigma$ and
$\alpha\circ\Sigma_\beta$ is a diffeomorphism.
\end{proposition}
\begin{proof}
Indeed, since $\map{\alpha|_\Sigma}{\Sigma}{M}$ is a
diffeomorphism, the map $\Sigma_\alpha$ is well defined and
satisfies $\alpha\circ\Sigma_\alpha=\id_M$. Moreover
$\Im\Sigma_\alpha=\Im((\alpha_\Sigma)^{-1})=\Sigma$ and
$\beta\circ\Sigma_\alpha=(\beta|_\Sigma)\circ(\alpha|_\Sigma)^{-1}$
is a diffeomorphism because it is a composition of
diffeomorphisms. The proof of the second statement is similar.
\end{proof}

Notice that the diffeomorphisms $\beta\circ\Sigma_\alpha$ and
$\alpha\circ\Sigma_\beta$ are each one the inverse of the
other
\[
(\beta\circ\Sigma_\alpha)^{-1}=\alpha\circ\Sigma_\beta
\qquad\text{and}\qquad
(\alpha\circ\Sigma_\beta)^{-1}=\beta\circ\Sigma_\alpha.
\]

\begin{definition}
A \emph{local bisection} of $G$ is a closed embedded submanifold
$\calw$ of $G$ such that there exist open subsets
$\calu,\calv\subset M$ for which both
$\map{\alpha|_\calw}{\calw}{\calu}$ and
$\map{\beta|_\calw}{\calw}{\calv}$ are diffeomorphisms.
\end{definition}

\begin{proposition}
Given a local bisection $\calw$ the map
$\calw_\alpha=(\alpha|_\calw)^{-1}$ is a smooth local section of
$\alpha$ defined on the open set $\calu$ such that
$(\calw_\alpha)(\calu)=\calw$ and
$\map{\beta\circ\calw_\alpha}{\calu}{\calv}$ is a diffeomorphism.
Alternatively, the map $\calw_\beta=(\beta|_\calw)^{-1}$ is a
smooth local section of $\beta$ defined on the open set $\calv$
such that $(\calw_\beta)(\calv)=\calw$ and
$\map{\alpha\circ\calw_\beta}{\calv}{\calu}$ is a diffeomorphism.
\end{proposition}
\begin{proof}
The proof is a straightforward modification of the proof for global bisections.
\end{proof}

We will need the following straightforward result.

\begin{lemma}
Let $A$ and $B$ be linear subspaces of a finite dimensional vector
space $V$, with $\dim(A)=\dim(B)$. There exists a linear subspace
$C\subset V$ such that $A\oplus C=V$ and $B\oplus C=V$.
\end{lemma}
\begin{proof}
Let $\{c_i\}$ be a basis of $A\cap B$. We complete to a basis
$\{c_i,a_\alpha\}$ of $A$, and we also complete to a basis
$\{c_i,b_\alpha\}$ of $B$. Then $\{c_i,a_\alpha,b_\alpha\}$ is a
basis of $A+B$, which can be completed to a basis
$\{c_i,a_\alpha,b_\alpha,d_J\}$ of $V$. Then, the subspace
$C=\operatorname{span}\{a_\alpha+b_\alpha,d_J\}$ is such that
$A\oplus C=V$ and $B\oplus C=V$.
\end{proof}

\begin{proposition}[Existence of local bisections]
Given $g\in G$ there exists a local bisection $\calw$ such that
$g\in\calw$.
\end{proposition}
\begin{proof}
Since the dimensions of $\ker(T_g\alpha)$ and $\ker(T_g\beta)$ are
equal, there exists a subspace $I\subset T_gG$ such that
$T_gG=\ker(T_g\alpha)\oplus I$ and $T_gG=\ker(T_g\beta)\oplus I$.
Notice that, since $\alpha$ and $\beta$ are submersions, we have
that $T_g\alpha (I)=T_{\alpha(g)}M$ and
$T_g\beta(I)=T_{\beta(g)}M$.

Let $\Sigma\subset G$ be any submanifold such that $g\in\Sigma$
and $T_g\Sigma=I$. The maps $T_g\alpha|_{\Sigma}$ and
$T_g\beta|_{\Sigma}$ are linear isomorphisms at the point $g$.
Indeed $\Im(T_g\alpha|_\Sigma)=T_g\alpha(I)=T_{\alpha(g)}M$, and
similarly $\Im(T_g\beta|_\Sigma)=T_g\beta(I)=T_{\beta(g)}M$, and
$\dim(M)=\dim(I)$. By the inverse function theorem, it follows
that there exist open subsets $\calw^\alpha$ and $\calw^\beta$ in
$\Sigma$ and $\calu^\alpha$ and $\calv^\beta$ in $M$ such that
$g\in\calw^\alpha\cap\calw_\beta$ and
$\map{\alpha|_{\calw^\alpha}}{\calw^\alpha}{\calu^\alpha}$ and
$\map{\beta|_{\calw^\beta}}{\calw^\beta}{\calv^\beta}$ are
diffeomorphisms. By taking $\calw=\calw^\alpha\cap\calw^\beta$,
$\calu=\alpha(\calw)$ and $\calv=\beta(\calw)$ we have that $\calw$
is a local bisection and $g\in\calw$.
\end{proof}

In what follows we do not distinguish, in the notation, global and
local bisections; all then will be denoted by $\Sigma$.

\begin{definition}
Given a local bisection $\Sigma$ defined on the open sets $\calu$
and $\calv$, the local \emph{left translation} by
$\Sigma$ is the map
$\map{L_\Sigma}{\alpha^{-1}(\calv)}{\alpha^{-1}(\calu)}$ defined by
\[
L_\Sigma(g)=hg,\qquad\text{where $h=\Sigma_\beta(\alpha(g))$},
\]
and the local \emph{right translation} by $\Sigma$ is the map
$\map{R_\Sigma}{\beta^{-1}(\calu)}{\beta^{-1}(\calv)}$ defined by
\[
R_\Sigma(g)=gh,\qquad\text{where $h=\Sigma_\alpha(\beta(g))$}.
\]
\end{definition}

Alternatively, the left action is $\Sigma\cdot g=hg$, where
$h\in\Sigma$ is the uniquely defined element such that
$\beta(h)=\alpha(g)$. Similarly the right action is $g \cdot
\Sigma=gh$, where $h\in\Sigma$ is the uniquely defined element
such that $\alpha(h)=\beta(g)$. Observe that both $L_\Sigma$ and $R_\Sigma$ are diffeomorphisms.

It is easy to see that, for a bisection $\Sigma$, the left action $L_\Sigma$ preserves the $\beta$-fibers and maps $\alpha$-fibers to $\alpha$-fibers. Similarly, the right action $R_\Sigma$ preserves the $\alpha$-fibers and maps $\beta$-fibers to $\beta$-fibers.

The left action $L_\Sigma$ by a bisection $\Sigma$ extends the natural left action in the groupoid and we have $(\Sigma\cdot g)h=\Sigma\cdot(gh)$, or in other words $L_\Sigma\circ l_g=l_{\Sigma\cdot g}$. Moreover, since it preserves the $\alpha$-fibers, it maps left-invariant vector fields to left invariant vector fields
\begin{equation}\label{left-invariant-bisection}
TL_\Sigma(\lvec{X}(g))=\lvec{X}(L_\Sigma\, g)
\qquad\text{for every $X\in\Gamma(\tau)$}.
\end{equation}
Similarly, we have $(hg)\cdot\Sigma=h(g\cdot\Sigma)$ or $R_\Sigma\circ r_g=R_{g\cdot \Sigma}$, from where
\begin{equation}\label{right-invariant-bisection}
TR_\Sigma(\rvec{X}(g))=\rvec{X}(R_\Sigma\, g)
\qquad\text{for every $X\in\Gamma(\tau)$}.
\end{equation}

\subsection{General discrete Euler-Lagrange evolution operators}
\label{GdE-Leo}
 Let $L_{d}: G \to {\Real}$ be a discrete Lagrangian
function and consider a solution $(g_{0}, h_{0}) \in G_{2}$ of the
discrete Euler-Lagrange equations, that is,
\[
\lvec{X}(g_{0})(L_{d}) - \rvec{X}(h_{0})(L_{d}) = 0, \; \; \mbox{
for all } X \in \Gamma(\tau).
\]
We denote by $x_0\in M$ the point
\[
x_0=\beta(g_{0}) = \alpha(h_{0}) .
\]
Then, we consider the following objects:
\begin{itemize}
\item
A symmetric neighborhood ${\mathcal W}$ associated with $x_{0}$
and some open subset $U$ of $G$, with local coordinates $(x, u)$
as in Section~\ref{local-expressions}. We will denote by $V$ the
corresponding open subset of $M$ and by $(y)$ the local
coordinates on $V$.
\item
A local bisection $\Sigma_{0}$ of $G$ such that $g_{0} \in
\Sigma_{0}$.
\item
A local bisection $\Upsilon_{0}$ of $G$ such that $h_{0} \in
\Upsilon_{0}$.
\end{itemize}

We will assume, without the loss of generality, that the section
$(\Sigma_{0})_{\beta}$ (respectively, $(\Upsilon_{0})_{\alpha}$)
is defined on the open subset $V$ of $M$.

It is clear that $\calw_{\Sigma_0} =  L_{\Sigma_{0}}(\calw)$ is an open
neighborhood of $g_0$ in $G$ diffeomorphic to $\calw$. On $\calw_{\Sigma_{0}}$ we may consider a local coordinate system $(x, u)$ defined as follows: if $g\in\calw_{\Sigma_0}$ then there exists a unique $g_\calw\in\calw$ such that $L_{\Sigma_0}(g_\calw)=g$; the coordinates of the point $g$ are the coordinates in the symmetric neighborhood $\calw$ of he point $g_\calw$.

Similarly, $\calw_{\Upsilon_0} =  R_{\Upsilon_{0}}(\calw)$ is an open
neighborhood of $h_0$ in $G$ diffeomorphic to $\calw$. On $\calw_{\Upsilon_{0}}$ we consider local coordinates $(y, v)$ defined as follows: for $h\in\calw_{\Upsilon_0}$ we consider the unique point $h_\calw\in\calw$ such that $R_{\Upsilon_0}(h_\calw)=h$, and we assign to $h$ the coordinates $(y,v)$ of $h_\calw$ in the symmetric neighborhood $\calw$.

Moreover, using the same notation as in Section~\ref{local-expressions}, the pair of elements $(g, h) = (L_{\Sigma_{0}}(g_{\mathcal W}), R_{\Upsilon_{0}}(h_{\mathcal W})) \in {\mathcal W}_{\Sigma_{0}}
\times {\mathcal W}_{\Upsilon_{0}}$ is composable if and only if $y=\lbeta(x,u)$.

%


To find the local equations satisfied by the coordinates of a solution $(g,h)$ of the discrete Euler-Lagrange equations for $L_d$
\begin{equation}\label{E-L1} \lvec{e_{\gamma}}(g)(L_{d}) -
\rvec{e_{\gamma}}(h)(L_{d}) = 0, \; \; \mbox{ for all } \gamma,
\end{equation}
we take into account equations~\eqref{left-invariant-bisection} and~\eqref{right-invariant-bisection}, from where we get that equations~\eqref{E-L1} hold if and only if
 \[
 (T_{g_{\mathcal W}}L_{\Sigma_{0}})(\lvec{e_{\gamma}}(g_{\mathcal
 W}))(L_{d}) - (T_{h_{\mathcal
 W}}R_{\Upsilon_{0}})(\rvec{e_{\gamma}}(h_{\mathcal W}))(L_{d}) = 0, \; \;
 \mbox{ for all } \gamma
 \]
 or, equivalently,
 \[
 \lvec{e_{\gamma}}(g_{\mathcal W})(L_{d} \circ L_{\Sigma_{0}}) -
 \rvec{e_{\gamma}}(h_{\mathcal W})(L_{d} \circ R_{\Upsilon_{0}}) = 0, \; \;
 \mbox{ for all } \gamma.
 \]
 In conclusion, using (\ref{left-alpha}) and (\ref{right-alpha}), we have proved
 \begin{theorem}\label{Local-DEL}
The pair $(g, h) = (L_{\Sigma_{0}}(g_{\mathcal W}),
R_{\Upsilon_{0}}(h_{\mathcal W})) \in {\mathcal W}_{\Sigma_{0}}
\times {\mathcal W}_{\Upsilon_{0}}$ is a solution of the discrete
Euler-Lagrange equations for $L_{d}$ if and only if the local
coordinates of $g$ and $h$
\[
g \cong (x, u) \cong g_{\mathcal W}, \; \; \; h \cong (y, v) =
(\lbeta(x, u), v) \cong h_{\mathcal W}
\]
satisfy the equations
\begin{multline}
L^\gamma_\mu(x,u)\pd{(L_d\circ
L_{\Sigma_{0}})}{u^\gamma}(x,u)
+\displaystyle\rho^i_\mu(y)\pd{(L_d\circ
R_{\Upsilon_{0}})}{x^i}(y,v) +\\
-R^\gamma_\mu(y,v)\pd{(L_d\circ R_{\Upsilon_{0}})}{u^\gamma}(y,v)=0.
\end{multline}
\end{theorem}

Note that the coordinates of $g_{0}$ and $h_{0}$ in ${\mathcal
W}_{\Sigma_{0}}$ and ${\mathcal W}_{\Upsilon_{0}}$, respectively,
are $(x_{0}, 0)$.

Next, as in Section \ref{Implicit-theorem}, we will consider the
matrix $(\mathbb{F}L_d)^{\gamma}_{\mu}(x, u)$, where
\begin{multline*}
(\mathbb{F}L_d)^{\gamma}_{\mu}(x, u) =
\rho^i_\mu(x)\frac{\partial^2(L_{d} \circ
R_{\Upsilon_{0}})}{\partial x^i\partial u^\gamma}(x,u)+\\
-\frac{\partial R^\nu_\mu}{\partial
u^\gamma}(x,u)\frac{\partial (L_d \circ R_{\Upsilon_{0}})}
{\partial u^\nu}(x,u) +\\
-R^\nu_\mu(x,u)\frac{\partial^2(L_d\circ
R_{\Upsilon_{0}})}{\partial u^\nu \partial u^\gamma}(x,u).
\end{multline*}

Then, we have the following result:
\begin{theorem}\label{DELE-operator1}
The following statements are equivalent.
\begin{itemize}
\item The matrix $(\mathbb{F}L_d)^{\gamma}_{\mu}(x_0,0)$ is regular.
\item The Poincare-Cartan 2-section $\Omega_{L_d}$ is non-degenerate at the point $h_0$.
\item The map $\mathbb{F}^-L_d$ is a local diffeomorphism at $h_0$.
\end{itemize}
Any of them implies the following: there exist open neighborhoods
$\widetilde{\mathcal W}_{\Sigma_{0}} \subseteq {\mathcal
W}_{\Sigma_{0}}$ and $\widetilde{\mathcal W}_{\Upsilon_{0}}
\subseteq {\mathcal W}_{\Upsilon_{0}}$ of $g_{0}$ and $h_{0}$ such
that if $g\in \widetilde{\mathcal W}_{\Sigma_{0}}$ then there is a
unique $\Psi(g) = h \in \widetilde{\mathcal W}_{\Upsilon_{0}}$
satisfying that the pair $(g, h)$ is a solution of the
Euler-Lagrange equations for $L_{d}$. In fact, the map $\Psi:
\widetilde{\mathcal W}_{\Sigma_{0}} \to \widetilde{\mathcal
W}_{\Upsilon_{0}}$ is a local discrete Euler-Lagrange evolution
operator.
\end{theorem}
\begin{proof}
If the matrix $(\mathbb{F}L_d)^{\gamma}_{\mu}(x_{0},0)$ is regular
then, using Theorem~\ref{Local-DEL} and the implicit function
theorem, we deduce the result about the existence of the local
discrete Euler-Lagrange evolution operator $\Psi:
\widetilde{\mathcal W}_{\Sigma_{0}} \to \widetilde{\mathcal
W}_{\Upsilon_{0}}$.

On the other hand, the map $R_{\Upsilon_{0}}$ is a
diffeomorphism from an open neighborhood of $\epsilon(x_{0})$ in the $\alpha$-fiber
$\alpha^{-1}(x_{0})$ to an open neighborhood of $h_{0}$ in the $\alpha$-fiber
$\alpha^{-1}(x_{0})$. Indeed, (a) it is well defined: if $h\in\alpha^{-1}(x_0)$ then
\[
\alpha(R_{\Upsilon_0}(h))=\alpha(h\cdot \Upsilon_0)=\alpha(h)=x_0,
\]
so that $R_{\Upsilon_0}(h)\in\alpha^{-1}(x_0)$; (b) it is injective: it is the restriction of a local diffeomorphism to an $\alpha$-fiber; and (c) it is surjective: $\beta\circ (\Upsilon_0)_\alpha$ is a diffeomorphism, so that if $h'\in\alpha^{-1}(x_0)$ there exists $x\in M$ such that $(\beta\circ (\Upsilon_0)_\alpha)(x)=\beta(h')$, from where $h=h'[(\Upsilon_0)_\alpha(x)]^{-1}\in\alpha^{-1}(x_0)$ and hence $R_{\Upsilon_0}(h)=h'$.

Consequently, if
\[
w^{\gamma}(h_{0}) =
(T_{\epsilon(x_{0})}R_{\Upsilon_{0}})(\displaystyle
\frac{\partial}{\partial u^{\gamma}}_{|\epsilon(x_{0})}), \mbox{
for all } \gamma,
\]
then $\{w^{\gamma}(h_{0})\}$ is a basis of the vertical bundle to
$\alpha$ at the point $h_{0}$.

Moreover, if $\rvec{e_{\mu}}$ is the right-invariant vector field
on ${\mathcal W}_{\Upsilon_{0}}$ given by
\[
\rvec{e_{\mu}}(h) = (T_{h_{\mathcal
W}}R_{\Upsilon_{0}})(\rvec{e_{\mu}}(h_{\mathcal W})), \; \; \mbox{
for } h = R_{\Upsilon_{0}}(h_{\mathcal W}) \in {\mathcal
W}_{\Upsilon_{0}}
\]
then, using (\ref{right-alpha}), it follows that
\[
\rvec{e}_\mu(L_d)\circ R_{\Upsilon_0}
=-\rho^i_\mu\pd{(L_d\circ R_{\Upsilon_0})}{x^i}+R^\nu_\mu\pd{(L_d\circ R_{\Upsilon_0})}{u^\nu}
\]
which implies that the matrix $(w^{\gamma}(h_{0})(\rvec{e_{\mu}}(L_{d})))$ is, up to the sign,
$(\mathbb{F}L_d)^{\gamma}_{\mu}(x_{0}, 0)$.

 In addition, from~\eqref{DLt-} and~\eqref{right-invariant-bisection}, we
deduce that the local expression of the restriction of
$\mathbb{F}^-{L_d}$ to ${\mathcal W}_{\Upsilon_{0}}$ is
\[
(\mathbb{F}^-{L_d})(x, u) = (x,
-\rho^{i}_{\gamma}(x)\displaystyle \frac{\partial (L_{d} \circ
R_{\Upsilon_{0}})}{\partial x^i}(x, u) + R^{\mu}_{\gamma}(x,
u)\displaystyle \frac{\partial (L_{d} \circ
R_{\Upsilon_{0}})}{\partial u^{\mu}}(x, u)).
\]
These facts prove the result.
\end{proof}

Now, as in Section \ref{Implicit-theorem}, we consider local
coordinates $(\bar{x}, \bar{u})$ on $i({\calu}) = \bar{\calu}$
given by $(\bar{x}, \bar{u}) = (x, u) \circ i$. Note that
${\mathcal W} \subseteq \bar{\calu}$ and, thus, we have local
coordinates on ${\mathcal W}_{\Sigma_{0}}$ and ${\mathcal
W}_{\Upsilon_{0}}$ which we also denote by $(\bar{x}, \bar{u})$.

As above, using equations~\eqref{inverse-right-alpha} and~\eqref{inverse-left-alpha},  we may prove that a pair $(g, h) =
(L_{\Sigma_{0}}(g_{\mathcal W}), R_{\Upsilon_{0}}(h_{\mathcal W}))
\in {\mathcal W}_{\Sigma_{0}} \times {\mathcal W}_{\Upsilon_{0}}$
is a solution of the discrete Euler-Lagrange equations for $L_{d}$
if and only if the local coordinates of $g$ and $h$
\[
g \cong (\bar{y}, \bar{v}) = (\lbeta(\bar{x}, \bar{u}), \bar{v})
\cong g_{\mathcal W}, \; \; \; h \cong (\bar{x}, \bar{u}) \cong
h_{\mathcal W}
\]
satisfy the equations
\begin{multline}
\label{E-L2}
0=-\rho^i_\mu(\lbeta(\bar{x},\bar{u}))\pd{(L_d\circ
L_{\Sigma_{0}})}{\bar{x}^i}(\lbeta(\bar{x},\bar{u}),\bar{v}) +\\
 +
R^\gamma_\mu(\lbeta(\bar{x},\bar{u}),\bar{v})\pd{(L_d\circ
L_{\Sigma_{0}})}{\bar{u}^\gamma}(\lbeta(\bar{x}, \bar{u}),\bar{v})+
\\
-  \displaystyle L^\gamma_\mu(\bar{x},\bar{u})\pd{(L_d\circ
R_{\Upsilon_{0}})}{\bar{u}^\gamma}(\bar{x},\bar{u}).
\end{multline}

In order to apply the implicit function theorem
(using~\eqref{E-L2}, we want to obtain $\bar{v}$ in terms of
$\bar{x}$ and $\bar{u}$ in a neighborhood of $(x_0,u_0,v_0)$), we
consider the matrix
$\overline{(\mathbb{F}L)}^{\gamma}_{\mu}(\bar{x}, \bar{u})$, where
\begin{multline}
\label{Implicit2}
\overline{(\mathbb{F}L)}^{\gamma}_{\mu}(\bar{x}, \bar{u})  =
\rho^i_\mu(\bar{x})\frac{\partial^2(L_{d} \circ
R_{\Upsilon_{0}})}{\partial \bar{x}^i\partial
\bar{u}^\gamma}(\bar{x},\bar{u})+\\
-\frac{\partial
R^\nu_\mu}{\partial \bar{u}^\gamma}(\bar{x},\bar{u})\frac{\partial
(L_d \circ R_{\Upsilon_{0}})}
{\partial \bar{u}^\nu}(\bar{x},\bar{u}) +\\
-L^\nu_\mu(\bar{x},\bar{u})\frac{\partial^2(L_d\circ
R_{\Upsilon_{0}})}{\partial \bar{u}^\nu \partial
\bar{u}^\gamma}(\bar{x},\bar{u}).
\end{multline}
 Then, one may prove the following result.

\begin{theorem}
The following statements are equivalent.
\begin{itemize}
\item The matrices $(\mathbb{F}L_d)^{\gamma}_{\mu}(x_0,0)$ and
$\overline{(\mathbb{F}L)}^{\gamma}_{\mu}(x_0,0)$ are regular.
\item The Poincare-Cartan 2-section $\Omega_{L_d}$ is non-degenerate
at the points $g_0$ and $h_{0}$.
\item The maps $\mathbb{F}^+{L_d}$ and $\mathbb{F}^-{L_d}$
are local diffeomorphism at $g_{0}$ and~$h_0$.
\end{itemize}
Any of them implies the following: there exist open neighborhoods
$\widetilde{\mathcal W}_{\Sigma_{0}} \subseteq {\mathcal
W}_{\Sigma_{0}}$ and $\widetilde{\mathcal W}_{\Upsilon_{0}}
\subseteq {\mathcal W}_{\Upsilon_{0}}$ of $g_{0}$ and $h_{0}$ and
a unique (local) discrete Euler-Lagrange evolution operator $\Psi:
\widetilde{\mathcal W}_{\Sigma_{0}} \to \widetilde{\mathcal
W}_{\Upsilon_{0}}$ such that $\Psi(g_{0}) = h_{0}$. In addition,
$\Psi$ is a diffeomorphism.
\end{theorem}
\begin{proof}
Suppose that the matrices $(\mathbb{F}L)^{\gamma}_{\mu}(x_0,0)$
and $\overline{(\mathbb{F}L)}^{\gamma}_{\mu}(x_0,0)$ are regular.
Then, the existence of the local discrete Euler-Lagrange evolution
operator $\Psi: \widetilde{\mathcal W}_{\Sigma_{0}} \to
\widetilde{\mathcal W}_{\Upsilon_{0}}$ is guaranteed by Theorem
\ref{DELE-operator1}. Moreover, using (\ref{E-L2}),
(\ref{Implicit2}) and the implicit function theorem, we deduce
that there exist open neighborhoods of $g_{0}$ and $h_{0}$ which
we will assume, without the loss of generality, that are
$\widetilde{\mathcal W}_{\Sigma_{0}}$ and $\widetilde{\mathcal
W}_{\Upsilon_{0}}$ and, in addition, there exists a smooth map
\[
\Xi: \widetilde{\mathcal W}_{\Upsilon_{0}} \to \widetilde{\mathcal
W}_{\Sigma_{0}}
\]
such that, for each $h\in \widetilde{\mathcal W}_{\Upsilon_{0}}$,
the pair $(\Xi(h), h)$ is a solution of the discrete
Euler-Lagrange equations for $L_{d}$.

Thus, from Theorem \ref{DELE-operator1}, we obtain that
\begin{equation}
\label{Diff}
\Psi \circ \Xi = id, \; \; \;  \Xi \circ \Psi = id,
\end{equation}
which implies that $\Psi$ is a diffeomorphism. Indeed, if $g\in\tilde{\calw}_{\Sigma_0}$ then there is a unique element $h$ in $\tilde{\calw}_{\Upsilon_0}$, namely $h=\Psi(g)$, such that $(g,h)$ is a solution of the discrete Euler-Lagrange equations for $L_d$. This proves~\eqref{Diff} and thus $\Psi$ is a diffeomorphism.

On the other hand, the map $L_{\Sigma_{0}}$ is a diffeomorphism
from an open neighborhood of $\epsilon(x_{0})$ in
$\beta^{-1}(x_{0})$ on an open neighborhood of $g_{0}$ in
$\beta^{-1}(x_{0})$. Using this fact, \eqref{DLt+},
\eqref{left-alpha}, \eqref{left-invariant-bisection} and
proceeding as in the proof of Theorem \ref{DELE-operator1}, we
deduce the result.
\end{proof}

\subsection{Application}
Let $L_d: G\to \R$ be a discrete regular Lagrangian. 
Starting with a symmetric neighborhood ${\mathcal W}$ of the Lie groupoid $G$,  we have local coordinates expressions for the discrete Euler-Lagrange equations assuming that we can solve it on the symmetric neighborhood. If this is not the case, we are forced to use bisections. 

To illustrate this, assume that $G$ is a connected  Lie group and ${\mathcal W}$ is a symmetric neighborhood of the identity ${\mathfrak e}\in G$.
Assume that there exists a family of points of $G$: $\{g_i\}_{1\leq i\leq m}$ and $\{h_j\}_{1\leq j\leq n}$, with $g_1=h_1={\mathfrak e}$ such that $\displaystyle{G=\bigcup_{i=1}^m g_i{\mathcal W}=\bigcup_{j=1}^n {\mathcal W}h_j}$ (for instance, if $G$ is compact this property is verified).

Given an initial point $g$, then there exists at least an integer  $I$, $1\leq I\leq m$ such that $g\in g_{I}{\mathcal W}$ ($g_I\in G$ is the bisection using our notation). 
Now, we try to find an integer $J$, $1\leq J\leq n$ such that there exits a solution $h\in {\mathcal W}h_J$ of the following equation defined on the symmetric neighborhood ${\mathcal W}$: 
\[
 (T_{
g_{I}^{-1}g}
L_{g_{I}})(\lvec{e_{\gamma}}(g_{I}^{-1}g))(L_{d}) - (T_{
hh_{J}^{-1}}R_{h_{J}})(\rvec{e_{\gamma}}(hh_{J}^{-1}))(L_{d}) = 0, \; \;
 \mbox{ for all } \gamma
 \]
 where $\{e_{\gamma}\}$ is a basis of the Lie algebra ${\mathfrak g}$. 
If we find this $J$, $1\leq J\leq n$,   we will say that the pair $(g,h)\in G_2$ is a solution of the discrete Euler-Lagrange equations for $L_d$. 

We will explore in a future paper these techniques working with points not included in symmetric neighborhoods or points not included in the neighborhoods where the retraction maps are local diffeomorphisms.

\section{Conclusions and future work}
In this paper we have studied the local description of discrete Mechanics. In Section \ref{section:symmetric},  we have found  a local description of the discrete Euler-Lagrange equations using  the notion of  symmetric neighborhood on  Lie groupoids. In Section~\ref{sec:EL-general}, we extend this construction for points outside of this type of neighborhoods using bisections. 
We expect that our results apply to a wide range of  numerical methods using  discrete variational calculus. 

On the other hand, this paper will also open the possibility to easily adapt our construction to other families of
geometric integrators derived from discrete Mechanics, as for instance,  forced or dissipative systems, holonomic constraints,
explicitly time-dependent systems \cite{mawest}, frictional contact \cite{PaMa}
nonholonomic constraints \cite{IMMM}, multisymplectic field theories \cite{MaPa}, discrete optimal control 
\cite{JiKoMa,Leok,ObOlMa}.


\begin{thebibliography}{99}
\let\\, \newcommand{\by}[1]{\textsc{\ignorespaces #1}\\}
  \newcommand{\title}[1]{\textsl{\ignorespaces #1}\\}
  \newcommand{\vol}[1]{{\bf{\ignorespaces #1}}}
  \newcommand{\info}[1]{\textrm{\ignorespaces #1}.}

    \bibitem{AbMa} \by{Abraham R, Marsden JE} \title{Foundations of
Mechanics} \info{(2nd. edition), Benjamin/Cummings, Reading,
Massachusetts, 1978}


\bibitem{Rabee}
\by{Bou-Rabee N,  Marsden J E}
\title{Hamilton-Pontryagin Integrators on Lie Groups: Introduction and Structure-Preserving Properties}
\info{Foundations of Computational Mathematics, \vol{9} (2)  (2009) 197--219, (2009)}

    \bibitem{CaWe} \by{Cannas da Silva A, Weinstein A} \title{Geometric
    Models for Noncommutative Algebras} \info{Berkeley Mathematics
    Lecture Notes series, 10 Amer. Math. Soc., Providence RI; Berkeley Center
    for Pure and Appl. Math. Berkeley, 1999}

    \bibitem{CoDaWe}
\by{Coste A, Dazord P, Weinstein A} \title{Grupo\"\i des
symplectiques} \info{Pub. D\'{e}p. Math. Lyon \vol{2/A} (1987),
1--62}



\bibitem{HaLuWa} \by{Hairer E, Lubich, Ch,  Wanner, G} \title{ Geometric numerical integration. Structure-preserving algorithms for ordinary differential equations} \info{Springer Series in Computational Mathematics \vol{31} Springer, Heidelberg, 2010}


  \bibitem{HiMa} \by{Higgins PJ, Mackenzie K} \title{Algebraic
constructions in the category of Lie algebroids} \info{J. Algebra
\vol{129} (1990), 194--230}

\bibitem{IMMM} \by{Iglesias D,  Marrero J C,  Mart{\'\i}n de Diego D,  Mart{\'\i}nez E} \title{  Discrete nonholonomic Lagrangian systems on Lie groupoids} \info{J. Nonlinear Sci. \vol{18} (3) (2008) 221--276}

\bibitem{JiKoMa} \by{ Jim\'enez F, Kobilarov K,  Mart{\'\i}n de Diego D}
\title{Discrete Variational Optimal Control} \info{to appear in Journal of Nonlinear Science}




\bibitem{LeMaMa}
  \by{de Le\'on M, Marrero JC, Mart{\'\i}nez E}
  \title{Lagrangian submanifolds and dynamics on Lie algebroids}
  \info{J. Phys. A: Math. Gen. \vol{38} (2005), R241--R308}


\bibitem{rat} \by{Lewis D,  Ratiu T, Simo J C, Marsden J E}
\title{The heavy top: a geometric treatment}
\info{Nonlinearity \vol{5}  (1)(1992),  1--48}

\bibitem{Leok} \by{Leok M} \title{Foundations of Computational Geometric Mechanics} \info{Control and Dynamical Systems, California Institute of Technology, 2004}


\bibitem{Mac}
  \by{Mackenzie K}
  \title{General Theory of Lie Groupoids and Lie Algebroids}
  \info{London Mathematical Society Lecture Note Series \vol{213},
  Cambridge University Press, 2005}

  \bibitem{MaMaMa}
  \by{Marrero JC, Mart\'{\i}n de Diego D and Mart\'{\i}nez E}
  \title{Discrete Lagrangian and Hamiltonian Mechanics on Lie groupoids}
  \info{Nonlinearity \vol{19} (2006), no. 6, 1313--1348. Corrigendum:
  Nonlinearity  \vol{19}  (2006), no. 12, 3003--3004}


\bibitem{MaPa} \by{Marsden J E, Patrick G W,  Shkoller S} \title{ Multisymplectic geometry, variational integrators, and nonlinear PDEs} \info{Comm. Math. Phys. \vol{199} (2) (1998), 351--395}

\bibitem{mawest} \by{Marsden, J E,  West, M} \title{Discrete Mechanics and variational integrators} \info{Acta Numer.\vol{10} (2001), 357--514}

  \bibitem{Ma} \by{Mart\'{\i}nez E}  \title{Lagrangian Mechanics
   on Lie algebroids} \info{Acta Appl. Math. \vol{67} (2001),
   295--320}



\bibitem{ObOlMa} \by{Ober-Bl{\"o}baum, S, Junge, O, Marsden, J} \title{Discrete Mechanics and optimal control: an analysis} \info{ESAIM Control Optim. Calc. Var. \vol{17} (2) (2011), 322--352}



\bibitem{PaMa} \by{Pandolfi A,  Kane C, Marsden J E, Ortiz M}
\title{Time-discretized variational formulation of non-smooth frictional contact}
\info{Internat. J. Numer. Methods Engrg. \vol{53} (2002) (8) 1801--1829} 

  \bibitem{We0} \by{Weinstein A} \title{Lagrangian Mechanics and
groupoids} \info{Fields Inst. Comm. \vol{7} (1996), 207--231}



\end{thebibliography}
\end{document}